\documentclass[11pt]{article}
\usepackage{amsfonts}
\usepackage{amsmath,amssymb}
\usepackage{amsthm}
\usepackage[mathscr]{euscript}
\usepackage{mathrsfs}
\usepackage{indentfirst}
\usepackage{color}\usepackage{enumitem}

\newtheorem{theorem}{Theorem}[section]
\newtheorem{lemma}[theorem]{Lemma}

\newtheorem{remark}[theorem]{Remark}
\newtheorem{proposition}[theorem]{Proposition}
\numberwithin{equation}{section}

\newcommand{\lbl}[1]{\label{#1}}
\allowdisplaybreaks

\newcommand{\be}{\begin{equation}}
\newcommand{\ee}{\end{equation}}
\newcommand\bes{\begin{eqnarray}} \newcommand\ees{\end{eqnarray}}
\newcommand{\bess}{\begin{eqnarray*}}
\newcommand{\eess}{\end{eqnarray*}}
\newcommand{\bbbb}{\left\{\begin{aligned}}
\newcommand{\nnnn}{\end{aligned}\right.}
\newcommand{\bea}{\begin{align*}}
\newcommand{\eea}{\end{align*}}

\newcommand\ep{\varepsilon}

\newcommand\dd{\displaystyle}

\newcommand\df{\dd\frac}

\newcommand\dx{{\rm d}x}
\newcommand\dy{{\rm d}y}

\newcommand\yy{\infty}

\newcommand\R{\mathbb{R}}
\newcommand\sk{\smallskip}

\begin{document}\thispagestyle{empty}
\setlength{\baselineskip}{16pt}

\begin{center}
 {\LARGE\bf Dynamics of nonlocal diffusion problems}\\[2mm]
 {\LARGE\bf  with a free boundary and a fixed boundary\footnote{This work was supported by NSFC Grants
11771110, 11971128}}\\[4mm]
  {\Large Lei Li, \ \  Mingxin Wang\footnote{Corresponding author. {\sl E-mail}: mxwang@hit.edu.cn}}\\[1.5mm]
{School of Mathematics, Harbin Institute of Technology, Harbin 150001, PR China}
\end{center}

\date{\today}

\begin{abstract}
In this paper, we first consider two scalar nonlocal diffusion problems with a free boundary and a fixed boundary. We obtain the global existence, uniqueness and longtime behaviour of solution of these two problems. The spreading-vanishing dichotomy and sharp criteria for spreading and vanishing are established. We also prove that accelerated spreading could happen if and only if a threshold condition is violated by kernel function. Then we discuss a classical Lotka-Volterra predator-prey model with nonlocal diffusions and a free boundary which can be seen as nonlocal diffusion counterpart of the model in the work of Wang (2014 J. Differential Equations \textbf{256}, 3365-3394).



\end{abstract}

\section{Introduction}
Recently the authors of \cite{CDLL} proposed the following nonlocal diffusion model with free boundaries to study the population dispersal
\bes\left\{\begin{aligned}\label{1.1}
&u_t=d\int_{g(t)}^{h(t)}J(x-y)u(t,y)\dy-du+f(t,x,u), & &t>0,~g(t)<x<h(t),\\
&u(t,x)=0,& &t>0, ~ x\notin(g(t),h(t)),\\
&h'(t)=\mu\int_{g(t)}^{h(t)}\int_{h(t)}^{\infty}
J(x-y)u(t,x)\dy\dx,& &t>0,\\
&g'(t)=-\mu\int_{g(t)}^{h(t)}\int_{-\infty}^{g(t)}
J(x-y)u(t,x)\dy\dx,& &t>0,\\
&h(0)=-g(0)=h_0>0,\;\; u(0,x)=u_0(x),& &|x|\le h_0,
 \end{aligned}\right.
 \ees
 where kernel function $J$ satisfies
 \begin{enumerate}[leftmargin=4em]
\item[{\bf(J)}]$J\in C(\mathbb{R})\cap L^{\yy}(\mathbb{R})$, $J\ge 0$, $J(0)>0,~\dd\int_{\mathbb{R}}J(x)\dx=1$, \ $J$\; is even, \ $\dd\sup_{x\in\R}J(x)<\infty$,
 \end{enumerate}
 and the growth term $f$ satisfies
  \begin{enumerate}
  \item[{\bf(F1)}]\, {\rm(i)}\, $f\in C(\overline{\mathbb{R}}^+\times\mathbb{R}\times\overline{\mathbb{R}}^+)$, $f(t,x,0)\equiv 0$, and $f(t,x,u)$ is locally Lipschitz continuous in $u\in\overline{\mathbb{R}}^+$, i.e., for any given $M, T>0$, there exists a constant $L(M,T)>0$ such that
 \[|f(t,x,u_1)-f(t,x,u_2)|\le L(M,T)|u_1-u_2|\]
for all $u_1, u_2\in [0, M]$, $t\in [0, T]$ and all $x\in \R$;
\item[{\rm(ii)}]\, there exists $K>0$ such that for all $(t,x)\in \overline{\mathbb{R}}^+\times \R$, $f(t,x,u)<0$ when $u>K$.
 \end{enumerate}

 This model is mainly derived from two aspects: dispersal term and free boundary condition.
 As in \cite{HMV,AMRT}, we assume that $u(t,x)$ is the population density at point $x$ and time $t$, and $J(x-y)$ is the probability distribution function of jumping from location $y$ to $x$. Thus $J(x-y)u(t,y)$ is the rate at which individuals reach $x$ from position $y$, and by integrating over the entire survival area of the species we see that $\int_{g(t)}^{h(t)}J(x-y)u(t,y)\dy$ is the rate at which individuals arrive at $x$ from all other locations.  Similarly, $J(x-y)u(t,x)$ is the rate at which individuals leave location $x$ to jump to site $y$, and $u(t,x)=\int_{-\yy}^{\yy}J(x-y)u(t,x)\dy$ is the rate at which the individuals depart from location $x$ to go to all other positions. Taking $d$ as dispersal coefficient, we get the dispersal term
  \bess
 d\left(\int_{g(t)}^{h(t)}J(x-y)u(t,y)\dy-u(t,x)\right).
 \eess

 For the meaning of free boundary condition, similar to the corresponding local diffusion model \cite{DL2010,BDK}, we assume that the expanding rate of free boundary is proportional to the outward flux at the boundary. Please refer to \cite{CDLL} for more related details.

Based on the above analysis, let us now introduce our first model. Suppose that the initial habitat of the species is the spatial domain $[0,h_0)$. It is well-known to us that the species will instinctively expand their own habitat for the sake of their survival. We assume that the species can move to the area $(-\yy,0)$, but once they enter this area, they will die immediately, which means that this region is highly lethal to the species. Thus the species can only enlarge their habitat through the right boundary, and our model can be formulated by the following problem
 \bes\label{1.2}
\left\{\begin{aligned}
&u_t=d\dd\int_{0}^{h(t)}J(x-y)u(t,y)\dy-du+f(t,x,u), && t>0,~0\le x<h(t),\\
&u(t,h(t))=0,&& t>0,\\
&h'(t)=\mu\dd\int_{0}^{h(t)}\int_{h(t)}^{\infty}
J(x-y)u(t,x)\dy\dx,&& t>0,\\
&h(0)=h_0,\;\; u(0,x)=u_0(x),&& x\in[0,h_0],
\end{aligned}\right.
 \ees
 where $d,\mu$ and $h_0$ are positive constants, kernel function $J$ satisfies {\bf(J)}, and $u_0(x)$ satisfies
  \begin{enumerate}[leftmargin=4em]
\item[{\bf(H)}]$u_0\in C([0,h_0])$, \ $u_0>0$ in $[0,h_0)$, \ $u_0(h_0)=0$.
 \end{enumerate}

 Unlike problem \eqref{1.2}, in our second model we assume that the species know that the area $(-\yy,0)$ is seriously fatal to them, so they will not jump to this area. This indicates that the condition imposed at $x=0$ is analogous to the usual homogeneous Neumann boundary condition. Hence we can derive the following model
 \bes\label{1.3}
\left\{\begin{aligned}
&u_t=d\dd\int_{0}^{h(t)}J(x-y)u(t,y)\dy-d\dd\left(\int_{0}^{\yy}J(x-y)\dy\right) u+f(t,x,u), && t>0,~0\le x<h(t),\\
&u(t,h(t))=0, && t>0,\\
&h'(t)=\mu\dd\int_{0}^{h(t)}\int_{h(t)}^{\infty}
J(x-y)u(t,x)\dy\dx, && t>0,\\
&h(0)=h_0,\;\; u(0,x)=u_0(x), &&x\in[0,h_0],
\end{aligned}\right.
 \ees
 where $J$ and $u_0$ satisfy {\bf(J)} and {\bf(H)}, respectively.

 At last, we take the classical Lotka-Volterra predator-prey model as an example to study the interaction between two species, namely, we consider the following problem
 \bes\label{1.4}
\left\{\begin{aligned}
&u_{1t}=d_1\dd\int_{0}^{h(t)}\!\!J_1(x-y)u_1(t,y)\dy-d_1u_1+u_1(a_1-b_1u_1-c_1u_2), &&t>0,~0\le x<h(t),\\
&u_{2t}=d_2\dd\int_{0}^{h(t)}\!\!J_2(x-y)u_2(t,y)\dy-d_2u_2+u_2(a_2-b_2u_2+c_2u_1), &&t>0,~0\le x<h(t),\\
&u_i(t,h(t))=0, &&t> 0,\\
&h'(t)=\sum_{i=1}^2\dd\mu_i \int_{0}^{h(t)}\!\int_{h(t)}^{\infty}
J_i(x-y)u_i(t,x)\dy\dx, &&t> 0,\\
&h(0)=h_{0},\;\; u_i(0,x)=u_{i0}(x), &&0\le x\le h_0, ~ i=1,\,2,
\end{aligned}\right.
 \ees
where all parameters are positive, $J_i$ satisfy the condition {\bf(J)}, and $u_{i0}$ meet the condition {\bf(H)}.

For these nonlocal diffusion problems, there are two main differences from the corresponding local diffusion free boundary problems \cite{DL2010,Wjde}. Firstly, there is usually no regularity effect. The lack of regularity makes it difficult to derive some important uniform estimates which are crucial to study the dynamics for these models. Secondly, it is easily seen from the above equations that the change rate of species density at $(t,x)$ is affected not only by the density of species at $(t,x)$, but also by the value near site $x$, which leads to some difficulties and differences from local diffusion model when considering the longtime behaviour of solution. It is worthy mentioning that to overcome difficulties caused by the above differences, some new techniques are recently introduced in \cite {DWZ}.

Before starting our research, we now give a brief introduction for the recent works on nonlocal diffusion free boundary problems. By considering a semi-wave problem, Du et.al   \cite{DLZ} proved that when spreading happens, problem \eqref{1.1} has a finite spreading speed if and only if $J$ satisfies
  \begin{enumerate}[leftmargin=4em]
\item[{\bf(J1)}]$\dd\int_{0}^{\yy}xJ(x)\dx<\yy$.
 \end{enumerate}
Then Du and Ni \cite{DN21} gave comprehensive and delicate results on spreading speed by discussing the conditions satisfied by $J$. In \cite{DN200}, the authors showed that the local diffusion free boundary problem in one dimension space can be approximated by a slightly modified version of nonlocal diffusion free boundary problem in \cite{CDLL}. In addition, there are many other related works for nonlocal diffusion free boundary problems, such as \cite{DWZ,DN20} for two species models, \cite{DN213} for high dimension and radial symmetry version, \cite{WW1,WW2,LSW} for systems with nonlocal and local diffusions, \cite{LWW20} for two species models with different free boundaries, \cite{ZLZ} for logistic model in time periodic environment and \cite{ZLD,ZZLD} for the epidemic model with partial degenerate diffusion.

This paper is arranged as follows. In Section 2, some preparatory results are given, such as  eigenvalue problem, steady state problems on half space and a technical lemma; Section 3 is devoted to dynamics of the model \eqref{1.2}, and we prove the well-posedness, spreading-vanishing dichotomy and spreading or vanishing criteria, which also are obtained for the model \eqref{1.3} in Section 4. Moreover, the spreading speed for model \eqref{1.2} is discussed in Section 3. More precisely, for the model \eqref{1.2}, we prove that its spreading speed has lower and upper bounds if {\bf(J1)} holds, and accelerated spreading happens if {\bf(J1)} is violated. As for the model \eqref{1.3}, we show that it has a finite spreading speed if and only if {\bf(J1)} holds in Section 4; Section 5 is concerned with two species model \eqref{1.4}. The well-posedness, longtime behavior and the estimates of spreading speed are given. At last, a brief discussion is stated.

In this paper, we assume that the function $J$ satisfies the condition {\bf(J)}, and $0<T<\yy$.

\section{Some preparatory results}
\setcounter{equation}{0} {\setlength\arraycolsep{2pt}

In order to save space, we let $\mathbb{R}^+=(0,\yy)$ and $\overline{\mathbb{R}}^+=[0,\yy)$, and define
 \bess
 \mathbb H_{h_{0}}^T&=&\left\{h\in C([0,T]):\, h(0)=h_{0},
\; \inf_{0\le t_1<t_2\le T}\df{h(t_2)-h(t_1)}{t_2-t_1}>0\right\},
 \eess
and for $h\in\mathbb H_{h_{0}}^T$, we define
 \bess
 &D^T_{h}=\left\{(t,x):\, 0<t\leq T,~0\le x<h(t)\right\},\;\;\; D^\yy_{h}=\left\{(t,x):\, 0\le t<\yy,~0\le x\le h(t)\right\}
 \eess
and $\overline{D}^T_{h}$ is the closure of $D^T_h$.
In this section we first study the maximum principle and comparison principle, and then investigate the corresponding eigenvalue problem to \eqref{1.3}. Finally, we discuss the steady state problems corresponding to problems \eqref{1.2}, \eqref{1.3} and \eqref{1.4} respectively.

\subsection{The maximum principle and comparison principle}

 \begin{lemma}[Maximum principle]\label{l2.1}
Let $h\in\mathbb H^T_{h_0}$, $c\in L^\infty (D^T_h)$ and $\psi,\,\psi_t\in C(\overline{D}_h^T)$ satisfy
 \bess\left\{\begin{aligned}
&\psi_t\ge d\int_{0}^{h(t)}J(x-y)\psi(t,y)\dy+c\psi, \;\;\;\;\; 0<t\le T,\ 0\le x<h(t),\\
&\psi(t, h(t)) \geq 0,\;\;\;\;\;0<t\le T;\;\;\;\ \ \psi(0,x)\ge0,\;\;\;\;\;0\le x\le h_0.
 \end{aligned}\right.\eess
Then $\psi\ge0$ in $\overline{D}_h^T$. Moreover, if $\psi(0,x)\not\equiv0$ in $[0, h_0]$, then $\psi>0$ in $D_h^T$.
\end{lemma}
\begin{proof}Let $\Psi=\psi{\rm e}^{-kt}$ with constant $k>\|c(t,x)\|_{L^{\infty}(D_h^T)}+d$. Assume on the contrary that there exists $(t^*,x^*)\in (0,T]\times [0,h(t))$ such that $\Psi(t^*,x^*)=\min_{\overline{D}_h^T}\Psi<0$. Then $\Psi_t(t^*,x^*)\le0$, and
\[d\int_{0}^{h(t^*)}J(x^*-y)\Psi(t^*,y)\dy-d\int_{0}^{h(t^*)}J(x^*-y)\dy\Psi(t^*,x^*)\ge 0.\]
Therefore,
 \bess
0&\ge& \Psi_t(t^*,x^*)-d\int_{0}^{h(t^*)}J(x^*-y)\Psi(t^*,y)\dy+d\int_{0}^{h(t^*)}J(x^*-y)\dy\Psi(t^*,x^*)\\
&\ge&d\int_{0}^{h(t^*)}J(x^*-y)\dy\Psi(t^*,x^*)+c(t^*,x^*)\Psi(t^*,x^*)-k\Psi(t^*,x^*)>0.
 \eess
This is a contradiction, and so $\psi\ge0$ in $\overline{D}_h^T$. If $\psi(0,x)\not\equiv0$ in $[0, h_0]$, we can prove that $\psi>0$ in $D^T$ by similar arguments in the proof of \cite[Lemma 2.2]{CDLL}. The details are omitted here.
\end{proof}

\begin{lemma}[Maximum principle]\label{l2.2}
Let $v,c\in L^{\yy}([0,T]\times\mathbb{R}^+)$, $v, v_t\in C([0,T]\times\overline{\mathbb{R}}^+)$ and satisfy
 \bess\left\{\begin{aligned}
&v_t\ge d\int_{0}^{\yy}J(x-y)v(t,y)\dy+cv, && 0<t\le T,\ x\ge0,\\
&v(0,x)\ge0,  && x\ge0.
 \end{aligned}\right.\eess
Then $v\ge0$ in $[0,T]\times\overline{\mathbb{R}}^+$. Moreover, if $v(0,x)\not\equiv0$ in $\overline{\mathbb{R}}^+$, then $v>0$ in $(0,T]\times\overline{\mathbb{R}}^+$.
\end{lemma}

\begin{proof} Here we adopt an ingenious method introduced by Du and Ni \cite[Lemma 2.3]{DN21} to prove our conclusion. For any $\ep>0$, let $w=v+\ep {\rm e}^{At}$ with constant $A>d+1+\|c\|_{\yy}$. We prove
  \bes \label{2.1}w\ge0 \; \;{\rm ~ in ~ } \; [0,T]\times \overline{\mathbb{R}}^+.
  \ees
Define $T_0=\sup\big\{s\in (0,T]: w>0 \;{\rm ~ in} \; [0,s]\times \overline{\mathbb{R}}^+\big\}$. Since we have, for $(t,x)\in(0,T]\times\overline{\mathbb{R}}^+$,
  \bess
w_t-d\int_{0}^{\yy}J(x-y)w(t,y)\dy-cw&\ge& \big(A-c-d\int_{0}^{\yy}J(x-y)\dy\big)\ep {\rm e}^{At}>\ep {\rm e}^{At},
\eess
$w_t$ has a finite lower bound independent of $x\in\overline{\mathbb{R}}^+$. This together with the fact that $w(0,x)\ge\ep$ for $x\in\overline{\mathbb{R}}^+$ indicates that $T_0>0$ is well defined. If $T_0=T$, then \eqref{2.1} holds. Otherwise we have $T_0<T$ and $\inf_{x\in\overline{\mathbb{R}}^+}w(T_0,x)=0$ since if $\inf_{x\in\overline{\mathbb{R}}^+}w(T_0,x)>0$, because $w_t(T_0,x)$ has a finite lower bound independent of $x\in\overline{\mathbb{R}}^+$, we can find some $0<\delta\ll1$ such that $w(t,x)>0$ for $(t,x)\in [T_0,T_0+\delta]\times \overline{\mathbb{R}}^+$ which is a contradiction to the definition of $T_0$. Thus for any $0<\ep_n\ll1$ with $\ep_n\to 0$ as $n\to\yy$, there exist $x_n\in\overline{\mathbb{R}}^+$ such that $w(T_0,x_n)<\ep_n$.

Since $w$ is bounded below and continuous in $[0,T_0]\times \overline{\mathbb{R}}^+$, by Ekeland's variational principle, for these $\ep_n$, $(T_0,x_n)$ and $\sigma=\min\{T_0/2,1\}$, there exist $(t_n,y_n)\in[0,T_0]\times \overline{\mathbb{R}}^+$ such that
\[\begin{cases}
  w(t_n,y_n)\le w(T_0,x_n)<\ep_n, \;\; |T_0-t_n|+|x_n-y_n|<\sigma, \\[1mm]
  w(t_n,y_n)-w(t,y_n)\le \frac{|t_n-t|\ep_n}{\sigma} \;\;\; \mbox{for} ~ t\in[0,T_0].
\end{cases}\]
Due to the choice of $\sigma$, we have $t_n>0$, and thus
\bes\label{2.2}
w_t(t_n,y_n)\le\ep_n/\sigma\to0 \;\; {\rm ~ as} \;\; n\to\yy.
\ees Moreover, by $w\ge0$ in $[0,T_0]\times \overline{\mathbb{R}}^+$, we have
\bess
w_t(t_n,y_n)\ge c(t_n,y_n)w(t_n,y_n)+\ep {\rm e}^{At_n}
\ge-\|c\|_{\yy}\ep_n+\ep {\rm e}^{At_n}\ge \frac{1}{2}\ep ~ ~ {\rm for ~ all ~ large ~} n,
\eess
which contradicts to \eqref{2.2}. So \eqref{2.1} holds for any $0<\ep\ll1$. By letting $\ep\to0$ we derive $v\ge0$ in $[0,T]\times\overline{\mathbb{R}}^+$. If $v(0,x)\not\equiv0$, we can easily get the desired result by \cite[Lemma 3.3]{CDLL}.
\end{proof}

By the maximum principles obtained above we can get the standard comparison principles.

\subsection{Eigenvalue problem corresponding to \eqref{1.3}}

In the following, we discuss the principal eigenvalue of the following eigenvalue problem
\bes\label{2.3}
(\mathcal{L}^N_{(0,\,l)}+a_0) \phi:=d\int_{0}^{l}J(x-y)\phi(y)\dy-dj(x)\phi+a_0\phi=\lambda\phi ~ ~ \mbox{in}\;\;[0,l],
\ees
where $j(x)=\int_{0}^{\yy}J(x-y)\dy$ and $a_0,l$ are positive constants. It follows from {\bf(J)} that $j(x)$ is nondecreasing and continuously differentiable in $x\in\overline{\mathbb{R}}^+$, $j(0)=1/2$ and $j(x)>0$ for $x>0$. As
   \[\lim_{x\to0^+}\frac{x}{j(x)-1/2}=\frac{1}{J(0)}>0, \]
it yields that
\[\int_{0}^{l}\frac{1}{j(x)-1/2}\dx=\yy ~ ~ {\rm for ~ all ~ } l>0.\]
In view of \cite[Theorem 2.1]{HCV}, problem \eqref{2.3} has a principal eigenvalue $\lambda_p(\mathcal{L}^N_{(0,\,l)}+a_0)$ defined by
 \bes
\lambda_p(\mathcal{L}^N_{(0,\,l)}+a_0):=\inf\{\lambda\in\mathbb{R}: (\mathcal{L}^N_{(0,\,l)}+a_0) \phi\le\lambda\phi \mbox{ in $[0, l]$ for some } 0<\phi\in C([0, l])\}.\lbl{2.x}\ees

\begin{lemma}\label{l2.3}
  Assume that $l>0$. Then the followings hold true:

 \sk{\rm(1)}\, $\lambda_p(\mathcal{L}^N_{(0,\,l)}+a_0)$ is strictly increasing and continuous in $l$;

 \sk{\rm(2)}\, $\lim_{l\to\yy}\lambda_p(\mathcal{L}^N_{(0,\,l)}+a_0)=a_0$;

  \sk{\rm(3)}\, $\lim_{l\to0}\lambda_p(\mathcal{L}^N_{(0,\,l)}+a_0)=a_0-d/2$.
\end{lemma}

\begin{proof} (1)\, Since the continuity of $\lambda_p(\mathcal{L}^N_{(0,\,l)}+a_0)$ in $l$ can be proved by a similar method in \cite[Proposition 3.4]{CDLL}, we only prove the monotonicity.

Let $l_2>l_1>0$, and $\phi_1$ be the corresponding positive eigenfunction to  $\lambda_p(\mathcal{L}^N_{(0,\,l_1)}+a_0)$. Then
 \[\big[dj(x)-a_0+\lambda_p(\mathcal{L}^N_{(0,\,l_1)}+a_0)\big]\phi_1(x)
 =d\int_{0}^{l_1}J(x-y)\phi_1(y)\dy>0\;\;\mbox{in}\;\;[0,l_1],\]
which implies $dj(x)-a_0+\lambda_p(\mathcal{L}^N_{(0,\,l_1)}+a_0)>0$ in $[0,l_1]$. Together with the monotonicity of $j(x)$, we can define the following nonnegative continuous function
\begin{align*}
\tilde{\phi}(x)=\left\{\begin{aligned}
&\phi_1(x),& &x\in[0,l_1],\\
&\frac{d\dd\int_{0}^{l_1}J(x-y)\phi_1(y)\dy}{dj(x)-a_0+\lambda_p(\mathcal{L}^N_{(0,\,l_1)}+a_0)},& &x\in[l_1,l_2].
\end{aligned}\right.
  \end{align*}
It then follows that
  \[d\dd\int_{0}^{l_2}J(x-y)\tilde\phi(y)\dy-dj(x)\tilde\phi
  +a_0\tilde\phi\ge\lambda_p(\mathcal{L}^N_{(0,\,l_1)}+a_0)\tilde\phi\;\; {\rm ~ in ~ }\;[0,l_2],  \]
and this inequality holds strictly in $x=l_1$. Combing with $\tilde{\phi}(l_1)>0$, by the variational characterization of $\lambda_p(\mathcal{L}^N_{(0,\,l)}+a_0)$, we have
  \bess
\lambda_p(\mathcal{L}^N_{(0,\,l_2)}+a_0)&=&\sup_{0\ne\psi\in L^2([0,\,l_2])}\frac{d\dd\int_{0}^{l_2}\!\!\int_{0}^{l_2}J(x-y)\psi(y)\psi(x)\dy\dx
-d\dd\int_{0}^{l_2}\!j(x)\psi^2(x)\dx}{\dd\int_{0}^{l_2}\psi^2(x)\dx}+a_0\\
&\ge&\frac{d\dd\int_{0}^{l_2}\!\!\int_{0}^{l_2}J(x-y)\tilde{\phi}(y)\tilde{\phi}(x)\dy\dx
-d\dd\int_{0}^{l_2}\!j(x)\tilde{\phi}^2(x)\dx}{\dd\int_{0}^{l_2}\tilde{\phi}^2(x)\dx}+a_0\\
 &>&\lambda_p(\mathcal{L}^N_{(0,\,l_1)}+a_0).
\eess

{\rm(2)}\, Since $(\mathcal{L}^N_{(0,\,l)}+a_0) 1\le a_0$, it follows that $\lambda_p(\mathcal{L}^N_{(0,\,l)}+a_0)\le a_0$ for all $l>0$ by taking $\lambda=a_0$ and the test function $\phi\equiv 1$ in \eqref{2.x}. On the other hand, due to the condition {\rm \bf (J)}, for any given $0<\ep\ll 1$, there exists $L>0$ such that $\int_{-L}^{L}J(x)\dx>1-\ep$. For any $l>L$, taking $\phi\equiv 1$ as the test function in the variational characterization of $\lambda_p(\mathcal{L}^N_{(0,\,l)}+a_0)$ we have
 \bess
 \lambda_p(\mathcal{L}^N_{(0,\,l)}+a_0)&\ge&\frac{d\dd\int_{0}^{l}\int_{0}^{l}J(x-y)\dy\dx-d\dd\int_{0}^{l}j(x)\dx}{l}+a_0\\
 &\ge& \frac{d\dd\int_{L}^{l-L}\!\!\int_{0}^{l}J(x-y)\dy\dx}{l}-\frac{d\dd\int_{0}^{l}j(x)\dx}{l}+a_0\\
 &\ge& \frac{d(l-2L)(1-\ep)}{l}-\frac{d\dd\int_{0}^{l}j(x)\dx}{l}+a_0\to-\ep d+a_0 {\rm ~as ~} l\to\yy.
 \eess
By the arbitrariness of $\ep$, $\lim_{l\to\yy}\lambda_p(\mathcal{L}^N_{(0,\,l)}+a_0)\geq a_0$. The conclusion (2) holds.

\sk{\rm(3)}\, Let $\phi_1$ be the corresponding positive eigenfunction to  $\lambda_p(\mathcal{L}^N_{(0,\,l)}+a_0)$. Then
\bess
|\lambda_p(\mathcal{L}^N_{(0,\,l)}+a_0)-a_0+\frac{d}{2}|
&=&\dd\left|\frac{d\dd\int_{0}^{l}\int_{0}^{l}J(x-y)\phi_1(y)\phi_1(x)\dy\dx
-d\dd\int_{0}^{l}j(x)\phi^2_1(x)\dx}{\dd\int_{0}^{l}\phi^2_1(x)\dx}+\frac{d}{2}\right|\\
&\le&\frac{d\dd\int_{0}^{l}\int_{0}^{l}J(x-y)\phi_1(y)\phi_1(x)\dy\dx}{\dd\int_{0}^{l}\phi^2_1(x)\dx}+\frac{d\dd\int_{0}^{l}(j(x)-\frac{1}{2})\phi^2_1(x)\dx}{\dd\int_{0}^{l}\phi^2_1(x)\dx}\\
&\le&d\|J(x)\|_{\yy}l+d(a(l)-\frac{1}{2})\to 0 {\rm ~ as ~ } l\to 0.
\eess
The proof is complete. \end{proof}

\subsection{The steady state prolems}
In the sequel, we further assume that nonlinear term $f$ satisfies
 \begin{enumerate}
\item[{\bf(F2)}] $f\in C^1$ is independent of $(t,x)$, and $f(u)/u$ is strictly decreasing for $u>0$;
\item[{\bf(F3)}] $f'(0)>0$.
\end{enumerate}

From conditions {\bf(F1)}-{\bf(F3)}, we know that $f(u)=0$ has a unique positive root $u^*$. To study the longtime behavior of solution of problem \eqref{1.2}, we next give two lemmas about a steady state problem and a fixed boundary problem on half space respectively.

\begin{lemma}\label{l2.4}
Suppose that $f$ satisfies conditions {\bf(F1)}-{\bf(F3)}. Then the problem
  \bes\label{2.4}
  d\dd\int_{0}^{\yy}J(x-y)U(y)\dy-dU+f(U)=0 {\rm ~ ~ ~ in ~ }\;\overline{\mathbb{R}}^+
  \ees
has a unique bounded positive solution $U\in C(\overline{\mathbb{R}}^+)$, and $\lim_{x\to\yy}U(x)=u^*$. Moreover, $0<U<u^*$ and $U$ is non-decreasing in $\overline{\mathbb{R}}^+$.
\end{lemma}

\begin{proof}\, {\bf Step 1}: {\it Existence of bounded positive solution}. It is well known (e.g. see \cite[Proposition 3.5]{CDLL}) that for large $l>0$, the following problem
  \bes
  d\dd\int_{0}^{l}J(x-y)u(y)\dy-du+f(u)=0 {\rm ~ ~ ~ in ~ }\; [0,l]
  \lbl{x.3}\ees
has a unique positive solution $u_l\in C([0,l])$. If $u_l(x_0)=\max_{x\in[0,l]}u_l(x)>u^*$ for $x_0\in[0,l]$, then
  \[du_l(x_0)<du_l(x_0)-f(u_l(x_0))=d\dd\int_{0}^{l}J(x_0-y)u_l(y)\dy\le du_l(x_0),\]
which implies $0<u_l\le u^*$ for large $l$.

We claim that $u_l$ is non-decreasing in $l$ for large $l$. Since $u_l$ is positive and continuous in $[0,l]$, for any $l_2>l_1>0$, we can define
  \[\rho_*=\inf\{\rho>1: \rho u_{l_2}\ge u_{l_1}\; {\rm ~ in}\;[0,l_1]\}.\]
If $\rho_*>1$, it follows from the definition of $\rho_*$ that there exists $\bar{x}\in[0,l_1]$ such that $\rho_*u_{l_2}(\bar{x})=u_{l_1}(\bar{x})$.
Thanks to $ \rho_*u_{l_2}(x)\ge u_{l_1}(x)$ in $[0,l_1]$, we have
  \[d\dd\int_{0}^{l_1}J(\bar x-y)\rho_*u_{l_2}(y)\dy-d\rho_*u_{l_2}(\bar{x})+f(\rho_*u_{l_2}(\bar{x}))\ge0.\]
Meanwhile, we easily see that
  \[d\dd\int_{0}^{l_1}J(\bar x-y)\rho_*u_{l_2}(y)\dy-d\rho_*u_{l_2}(\bar{x})+\rho_*f(u_{l_2}(\bar{x}))\le0.\]
Hence $f(\rho_*u_{l_2}(\bar{x}))\ge \rho_*f(u_{l_2}(\bar{x}))$. Since $\rho_*>1$ and $u_{l_2}(\bar{x})>0$, one can derive a contradiction from the condition {\bf(F2)}. So our claim is true. Hence we can define $U(x)=\lim_{l\to\yy}u_l(x)$ for $x\in\overline{\mathbb{R}}^+$. Then $0<U\le u^*$. By the dominated convergence theorem, $U$ satisfies \eqref{2.4}.

{\bf Step 2}: {\it $\lim_{x\to\yy}U(x)=u^*$}. Arguing indirectly, there exist $\ep_1>0$ and $x_n\nearrow \yy$ such that $U(x_n)\le u^*-\ep_1$.
  Taking $w_n(x)=u_{2x_n}(x+x_n)$, then we have
  \[d\dd\int_{-x_n}^{x_n}J(x-y)w_n(y)\dy-dw_n(x)+f(w_n(x))=0 {\rm ~ ~ ~ for ~ }x\in[-x_n,x_n].\]
By \cite[Proposition 3.6]{CDLL}, we have $\lim_{n\to\yy}w_n(x)\to u^*$ locally uniformly in $\mathbb{R}$.  Thus there exists $N>0$ such that $u_{2x_n}(x_n)\ge u^*-\ep_1/2$ for any $n\ge N$, and we further have that $u^*-\ep_1\ge U(x_{N})\ge u_{2x_{N}}(x_{N})\ge u^*-\ep_1/2$. This is a contradiction, and so $\lim_{x\to\yy}U(x)=u^*$.

{\bf Step 3}: {\it The continuity of $U$}. Firstly, it is deduced from $J\in L^1(\mathbb{R})$ and $U\in L^{\yy}(\mathbb{R}^+)$ that $\int_{0}^{\yy}J(x-y)U(y)\dy$ is continuous in $\overline{\mathbb{R}}^+$. By virtue of $u_l\nearrow U$ and $\lim_{x\to\yy}U(x)=u^*$, we see that $U$ has a finite positive lower bound. Hence there exists a positive constant $C$ such that
  \bes\label{2.5}
  d-\frac{f(U(x))}{U(x)}\ge C>0, \;\;\;\forall\; x\in\overline{\mathbb{R}}^+.
  \ees
We shall show that for any $x_n\to x\in\overline{\mathbb{R}}^+$, there holds $U(x_n)\to U(x)$. Without loss of generality, we assume that $U(x_n)\ge U(x)$ for all $n$. Then, by \eqref{2.5}, we obtain
\bes
d\dd\int_{0}^{\yy}\big[J(x_n-y)-J(x-y)\big]U(y)\dy=\big[d-c(x)\big] [U(x_n)-U(x)]\ge C[U(x_n)-U(x)]\ge0,\quad\lbl{2.y}
\ees
where
\begin{align*}
c(x)=\left\{\begin{aligned}
&\frac{f(U(x_n))-f(U(x))}{U(x_n)-U(x)},& &U(x_n)\neq U(x),\\
&0,& &U(x_n)=U(x),
\end{aligned}\right.
  \end{align*}
and we have used the fact that $U(x_n)\ge U(x)$ implies   $\frac{f(U(x_n))-f(U(x))}{U(x_n)-U(x)}\le \frac{f(U(x))}{U(x)}$.
Since the left side of \eqref{2.y} goes to $0$ as $n\to\yy$, we have $U(x_n)\to U(x)$, and so $U\in C(\overline{\mathbb{R}}^+)$.

{\bf Step 4}: {\it Uniqueness}. Suppose that $V$ is another bounded positive solution of \eqref{2.4}. We claim $V\le u^*$. If $\sup_{x\in\overline{\mathbb{R}}^+}V(x):=\bar{V}>u^*$, then there exist $z_n\in\overline{\mathbb{R}}^+$ such that $V(z_n)\to \bar{V}$. Moreover,
  \[d\bar{V}<d\bar{V}-f(\bar{V})=\lim_{n\to\yy}d\dd\int_{0}^{\yy}J(z_n-y)V(y)\dy\le d\bar{V}.\]
Hence $V\le u^*$. We now prove that for any $l>0$, $V$ has a positive lower bound in $[0,l]$. If there exist $l_0>0$ and $\tilde x_n\to \tilde x_0\in[0,l_0]$ such that $V(\tilde x_n)\to 0$. Since $V$ satisfies \eqref{2.4}, we have
\[ d\dd\int_{0}^{l_0}J(\tilde x_n-y)V(y)\dy\le dV(\tilde x_n)+f(V(\tilde x_n)).\]
Letting $n\to\yy$, we derive that
\[\dd\int_{0}^{l_0}J(\tilde x_0-y)V(y)\dy\le0,\]
which implies
\[\dd\int_{0}^{l_0}J(\tilde x_0-y)V(y)\dy=0.\]
Hence $J(\tilde x_0-x)V(x)=0$ almost everywhere on $[0,l_0]$. By {\bf(J)} and positivity of $V$, we immediately obtain a contradiction. Thus $\inf_{x\in[0,l]}V(x)>0$ for all $l>0$. Then for large $l>0$, we can define
\[\eta_*=\inf\{\eta>1: \eta V(x)\ge u_l(x) ~ {\rm in} ~ [0,l]\}.\]
We claim $\eta_*=1$. Once our claim is true, we can show $u_l\le V$ for large $l>0$, and thus $U\le V$.
Assume that $\eta_*>1$. Then $\eta_* V(x)\ge u_l(x)$ in $[0,l]$. If there exists $\hat{x}_0\in[0,l]$ such that $\eta_* V(\hat{x}_0)=u_l(\hat{x}_0)$, then by arguing as in Step 1 we can get a contradiction. So $\eta_* V(x)>u_l(x)$ in $[0,l]$. If $\inf_{x\in[0,l]}\eta_* V(x)-u_l(x)=0$, then there exist $\bar{x}_n$ and $\bar{x}_0\in[0,l]$ such that $\bar{x}_n\to\bar{x}_0$ and $\eta_* V(\bar{x}_n)-u_l(\bar{x}_n)\to0$ as $n\to\yy$. Clearly, we have
\[ d\dd\int_{0}^{l}J(\bar x_n-y)\eta_*V(y)\dy-d\eta_*V(\bar x_n)+f(\eta_*V(\bar x_n))<0.\]
Therefore, by \eqref{x.3}, we deduce that
\bess\lim_{n\to\yy}d\dd\int_{0}^{l}J(\bar x_n-y)u_l(y)\dy&=&\lim_{n\to\yy}du_l(\bar x_n)-f(u_l(\bar{x}_n))=\lim_{n\to\yy}d\eta_*V(\bar{x}_n)-f(\eta_*V(\bar{x}_n))\\
&>&\lim_{n\to\yy} d\dd\int_{0}^{l}J(\bar x_n-y)\eta_*V(y)\dy,\eess
which yields
\[\dd\int_{0}^{l}J(\bar x_0-y)\left[u_l(y)-\eta_*V(y)\right]\dy\ge0.\]
This contradicts to $\eta_* V(x)>u_l(x)$ in $[0,l]$. Hence $\inf_{x\in[0,l]}\eta_* V(x)-u_l(x)>0$, and by the definition of $\eta_*$ and $\eta_*>1$ we easily show a contradiction. So we prove our claim.

Thanks to the above analysis, it is easily shown that $\lim_{x\to\yy}V(x)=u^*$ and $V$ is continuous in $\overline{\mathbb{R}}^+$. Similarly, we define
  \[\gamma_*=\inf\big\{\gamma>1: \gamma U\ge V {\rm ~ in ~ }\overline{\mathbb{R}}^+\big\}.\]
Obviously, $\gamma_*$ is well defined and $\gamma_*\ge1$. Next we show that $\gamma_*=1$. Once it is done, we have $U\ge V$, and so $U\equiv V$. Assume on the contrary that $\gamma_*>1$. By the definition of $\gamma_*$ and the fact that $\lim_{x\to\yy}U(x)=\lim_{x\to\yy}V(x)=u^*$, there must exist $\hat{x}\in\overline{\mathbb{R}}^+$ so that $\gamma_*U(\hat{x})=V(\hat{x})$. Similar to the above, it can be shown that $f(\gamma_*U(\hat{x}))\ge \gamma_*f(U(\hat{x}))$,  which is a contradiction to the assumptions of $f$. Hence $\gamma_*=1$,
and further $U\equiv V$.

{\bf Step 5}: {\it The proof of $U<u^*$ and the monotonicity of $U$}. If there exists $z_0\in\overline{\mathbb{R}}^+$ with $U(z_0)=u^*$, since $u^*$ is not a solution of \eqref{2.4}, then we can find some $\bar z\in \overline{\mathbb{R}}^+$ such that $U(\bar{z})=u^*$ and $U(x)<u^*$ in some left (right) neighborhood of $\bar{z}$. Furthermore, by \eqref{2.4},
   \[u^*=\int_{0}^{\yy} J(\bar{z}-y)U(y)\dy<u^*,\]
which implies that $U<u^*$.

Define $\bar U(x):=U(x+\delta)$ with any $\delta>0$. Then $\bar{U}$ satisfies
   \[d\dd\int_{0}^{\yy}J(x-y)\bar{U}(y)\dy-d\bar{U}+f(\bar{U})\le0 {\rm ~ ~ ~ in ~ }\; \overline{\mathbb{R}}^+.\]
   Similarly to the above, we can show that $U(x+\delta)=\bar{U}(x)\ge U(x)$ in $\overline{\mathbb{R}}^+$, and hence $U$ is nondecreasing in $\overline{\mathbb{R}}^+$. The proof is ended.
\end{proof}

\begin{lemma}\label{l2.5}\, Let $f$ satisfy conditions {\bf(F1)}-{\bf(F3)} and $U$ be the  unique bounded positive solution of \eqref{2.4}. Then, for any nonnegative function $w_0\in C(\overline{\mathbb{R}}^+)\cap L^{\yy}(\mathbb{R}^+)$, the following problem
 \bes\left\{\begin{aligned}\label{2.6}
&w_t= d\int_{0}^{\yy}J(x-y)w(t,y)\dy-dw+f(w), && t>0,\ x\ge0,\\
&w(0,x)=w_0(x)\ge,\,\not\equiv0,  && x\ge0
 \end{aligned}\right.\ees
has a unique solution $w$, and $\lim_{t\to\yy}w(t,x)=U(x)$ locally uniformly in $\overline{\mathbb{R}}^+$.
\end{lemma}

\begin{proof} \, We first show the existence and uniqueness of $w$.
Clearly, for large $l>0$, problem
 \bess\left\{\begin{aligned}
&w^l_t= d\int_{0}^{l}J(x-y)w^l(t,y)\dy-dw^l+f(w^l), && t>0,\ 0\le x\le l,\\
&w^l(0,x)=w_0(x),  && 0\le x\le l
 \end{aligned}\right.\eess
has a unique positive solution $w^l$. By the maximum principle, $0\le w^l\le \max\{\|w_0\|_{\yy}, K\}$, and $w^l$ is nondecreasing in $l$. Define $w=\lim_{l\to\yy}w^l$. Then $w(0,x)=w_0(x)$. For any $(t,x)\in \mathbb{R}^+\times \overline{\mathbb{R}}^+$, we have that, for large $l$,
 \[w^l(t,x)-w_0(x)= d\int_{0}^{t}\int_{0}^{l}J(x-y)w^l(\tau,y)\dy{\rm d}\tau-d\int_{0}^{t} w^l(\tau,x){\rm d}\tau+\int_{0}^{t} f(w^l(\tau,x)){\rm d}\tau.\]
Letting $l\to\yy$, it follows from the dominated convergence theorem that
 \[w(t,x)-w_0(x)= d\int_{0}^{t}\int_{0}^{\yy}J(x-y)w(\tau,y)\dy{\rm d}\tau-d\int_{0}^{t} w(\tau,x){\rm d}\tau+\int_{0}^{t} f(w(\tau,x)){\rm d}\tau.\]
This implies that $w$ satsisfies \eqref{2.6} and $w\in C(\overline{\mathbb{R}}^+\times\overline{\mathbb{R}}^+)$. The uniqueness can be deduced by the maximum principle.

Now we prove that longtime behaviour of $w$. Clearly, $\lim_{t\to\yy}w^l(t,x)=u_l(x)$ uniformly in $[0,l]$, which is given in the proof of Lemma \ref{l2.4}, and $\lim_{l\to\yy}u_l(x)\to U(x)$ locally uniformly in $\overline{\mathbb{R}}^+$ by Dini's theorem. Since $w\ge w^l$, we have $\liminf_{t\to\yy}w(t,x)\ge U(x)$ locally uniformly in $\overline{\mathbb{R}}^+$.

Let $W$ be the unique solution of \eqref{2.6} with initial data $W_0(x)=\max\{\|w_0\|_{\yy},K\}$. The maximum principle asserts that $W\ge w$ and $W$ is non-increasing in $t$. It follows from Lemma \ref{l2.4} and Dini's theorem that $\lim_{t\to\yy}W(t,x)=U(x)$ locally uniformly in $\overline{\mathbb{R}}^+$. So $\limsup_{t\to\yy}w(t,x)\le U(x)$ locally uniformly in $\overline{\mathbb{R}}^+$. The proof is ended.
\end{proof}
To carry out the discussion on longtime behavior of solution of problem \eqref{1.3}, we consider the fixed boundary problem
 \bes\left\{\begin{aligned}\label{2.7}
&u_t= d\int_{0}^{l}J(x-y)u(t,y)\dy-dj(x)u+f(u), && t>0,\ 0\le x\le l,\\
&u(0,x)=u_0(x)\geq 0,  && 0\le x\le l,
 \end{aligned}\right.\ees
where $j(x)=\int_{0}^{\yy}J(x-y)\dy$.

\begin{lemma}\label{l2.6}\, Let conditions {\bf(F1)}-{\bf(F3)} hold, $u_0\in C([0,l])$ and $u$ be the unique solution of \eqref{2.7}. Then we have
 \begin{enumerate}\vspace{-2mm}
\item[{\rm(1)}]\, problem \eqref{2.7} has a unique positive steady state $\tilde u_l\in C([0,l])$ if and only if $\lambda_p(\mathcal{L}^N_{(0,\,l)}+f'(0))>0$, and when $u_0\not\equiv 0$, $\lim_{t\to\yy}u=\tilde u_l$  in $C([0,l])$;\vspace{-2mm}
\item[{\rm(2)}]\, if $\lambda_p(\mathcal{L}^N_{(0,\,l)}+f'(0))\le0$, then $0$ is the only nonnegative bounded steady state of problem \eqref{2.7}, and $\lim_{t\to\yy}u=0$  in $C([0,l])$;\vspace{-2mm}
\item[{\rm(3)}]\, $\lim_{l\to\yy}\tilde u_l(x)=u^*$ locally uniformly in $\overline{\mathbb{R}}^+$, where $u^*$ is the unique positive root of $f(u)=0$.\vspace{-2mm}
    \end{enumerate}
\end{lemma}

\begin{proof}\, (1)\, Let $\tilde u_l$ be a positive steady state of \eqref{2.7}. Similarly to Lemma \ref{l2.4}, one can show that $\tilde u_l\in C([0,l])$. By the condition {\bf(F2)},
 \[d\int_{0}^{l}J(x-y)\tilde u_l(y)\dy-dj(x)\tilde u_l(x)+f'(0)\tilde u_l(x)>0 ~ ~ \mbox{in}\;\;[0,l].\]
Thanks to the variational characterization of $\lambda_p(\mathcal{L}^N_{(0,\,l)}+f'(0))$, we have $\lambda_p(\mathcal{L}^N_{(0,\,l)}+f'(0))>0$.

Suppose that $\lambda_p(\mathcal{L}^N_{(0,\,l)}+f'(0))>0$ with the corresponding positive eigenfunction $\phi$. Let $\beta=\max_{u\in[0,K]}|f'(u)|+d+1$, and define the operator $\Gamma$ by
 \[\Gamma(u)=\frac{1}{\beta}\left(d\int_{0}^{l}J(x-y)u(y)\dy-dj(x)u(x)+f(u)+\beta u\right)\;\;\;\mbox{for}\;\;u\in C([0,l]).\]
Obviously, $\Gamma$ is nondecreasing in $\{u\in C([0,l]):0\le u\le K\}$ and $\Gamma(K)\le K$. As $\lambda_p(\mathcal{L}^N_{(0,\,l)}+f'(0))>0$, we can find $0<\ep\ll 1$ such that $\Gamma(\ep\phi)\ge\ep\phi$ and $\ep\phi\le K$. Then by a simple iteration argument and the dominated convergence theorem, we can find a steady state of \eqref{2.7} denoted by $\tilde u_l$ with $\ep\phi\le \tilde u_l\le K$. Proof of the uniqueness and continuity is similar to that  of Lemma \ref{l2.4}.

Now we study the longtime behaviour of solution $u$ of \eqref{2.7}. Since $u_0\ge, \not\equiv 0$, the maximum principle gives $u(t,x)>0$ for $t>0$ and $0\le x\le l$. So, we may assume $u_0>0$ in $[0,l]$. There is $0<\ep\ll 1$ such that $\ep\phi\le u_0$ in $[0,l]$. Let $\underline u$ and $\bar u$ be the unique solution of \eqref{2.7} with initial data $\ep\phi$ and $\max\{K, \|u_0\|_{\yy}\}$, respectively. By the comparison principle, $\underline u\leq u\leq  \bar u$, and $\underline u$ and $\bar u$ are, respectively, nondecreasing and nonincreasing in $t$. It then follows from our early analysis that $\lim_{t\to\yy}\underline u=\lim_{t\to\yy}\bar u=\tilde u_l$ in $C([0,l])$. Hence $\lim_{t\to\yy}u=\tilde u_l$ in $C([0,l])$.

\sk{\rm(2)}\, Arguing indirectly, let $u_l\not\equiv0$ be a nonnegative steady state of \eqref{2.7}. Clearly, $u_l\in C([0,l])$, and $u_l>0$ in $[0,l]$ by the condition {\bf(J)}. Similarly to the above,
  \[d\int_{0}^{l}J(x-y)u_l(y)\dy-dj(x)u_l(x)+f'(0) u_l(x)>0~ ~ \mbox{in}\;\;[0,l].\]
This implies $\lambda_p(\mathcal{L}^N_{(0,\,l)}+f'(0))>0$. So, problem \eqref{2.7} has no nontrivial and nonnegative bounded steady state if  $\lambda_p(\mathcal{L}^N_{(0,\,l)}+f'(0))\le0$. Similar to the argument in (1),  $\lim_{t\to\yy}u=0$ in $C([0,l])$.

\sk{\rm(3)}\, Let $\tilde u_l$ be the unique positive steady state of \eqref{2.7}. By arguing as in the proof of Lemma \ref{l2.4}, we conclude that $\tilde u_l\le u^*$ and $\tilde u_l$ is nondecreasing in $l$. Then $\tilde{U}:=\lim_{l\to\yy}\tilde u_l$ exists and satisfies
 \[d\int_{0}^{\yy}J(x-y)\tilde{U}(y)\dy-dj(x)\tilde{U}+f(\tilde{U})=0.\]
Similarly to the proof of Lemma \ref{l2.4}, $\tilde{U}$ is continuous in $\overline{\mathbb{R}}^+$ and $\lim_{x\to\yy}\tilde{U}(x)=u^*$, and then we further have $\tilde{U}\equiv u^*$. Then by Dini's theorem, conclusion {\rm (3)} hlods.\end{proof}

The rest of this section is prepared for the study of longtime behavior of solution of \eqref{1.4}.

\begin{lemma}\label{l2.7}\, Suppose that $\lambda$ is a positive constant and the function $k$ satisfies
 \begin{enumerate}[leftmargin=4em]
\item[{\bf(K)}]\, $k\in C(\overline{\mathbb{R}}^+)\cap L^{\yy}(\mathbb{R}^+)$, $\underline k:=\inf_{x\in\overline{\mathbb{R}}^+}k(x)>0$ and $k_\yy :=\lim_{x\to\yy}k(x)<\yy$.
 \end{enumerate}
Consider problem
  \bes\label{2.8}
  d\dd\int_{0}^{\yy}J(x-y)U(y)\dy-du+U(k(x)-\lambda U)=0 {\rm ~ ~ ~ in ~ }\;\overline{\mathbb{R}}^+.
  \ees
Then the following results hold:
  \begin{enumerate}\vspace{-2mm}
\item[{\rm(1)}]\, problem \eqref{2.8} has a unique bounded positive solution $U_k\in C(\overline{\mathbb{R}}^+)$, and $\lim_{x\to\yy}U_k(x)={k_\yy}/{\lambda}$;\vspace{-2mm}
\item[{\rm(2)}]\, let functions $k_i$ satisfy the condition {\bf(K)} and $k_1\le k_2$ in $\overline{\mathbb{R}}^+$, then $U_{k_1}\le U_{k_2}$ in $\overline{\mathbb{R}}^+$.\vspace{-2mm}
    \end{enumerate}
\end{lemma}

\begin{proof} (1)\, We show the existence by a monotone iteration argument. By Lemma \ref{l2.4}, problem
  \[ d\dd\int_{0}^{\yy}J(x-y)u(y)\dy-du+u(\underline k-\lambda u)=0 {\rm ~ ~ ~ in ~ }\;\overline{\mathbb{R}}^+\]
has a unique positive solution $\underline{u}$ and $0<\underline{u}<\underline k/\lambda$.
Let $\bar{u}=\|k\|_{\yy}/\lambda$ and $\hat\beta=d+2\|k\|_{\yy}+1$. Similarly, the operator $\hat\Gamma$ defined by
  \[\hat\Gamma(u)=\frac{1}{\hat\beta}\left(d\int_{0}^{\yy}J(x-y)u(y)\dy-du+u(k(x)-\lambda u)+\hat\beta u\right), \ \ u\in C(\overline{\mathbb{R}}^+)\]
is nondecreasing in $\{u\in C(\overline{\mathbb{R}}^+):0\le u\le \bar{u}\}$, and $\hat\Gamma( \underline{u})\ge \underline{u}$ and $\hat\Gamma(\bar{u})\le \bar{u}$. Hence, by a iteration process, problem \eqref{2.8} has a positive solution $U_k$ with $\underline{u}\le U_k\le\bar{u}$. Since $d\int_{0}^{\yy}J(x-y)U_k(y)\dy$ and $k$ are continuous, it follows from a quadratic formula that $U_k$ is also continuous in $\overline{\mathbb{R}}^+$.

Now we prove $\lim_{x\to\yy}U_k(x)=\frac{k_\yy }{\lambda}$. Let   $u_*:=\liminf_{x\to\yy}U_k(x)$, then $u_*>0$ due to $U_k\ge\underline{u}$. If $u_*<k_\yy /\lambda$, then there exists $x_n\nearrow\yy$ such that $U_k(x_n)\to u_*$. Hence, as $n\to\yy$,
  \[d\int_{0}^{\yy}\!J(x_n-y)U_k(y)\dy=dU_k(x_n)-U_k(x_n)(k(x_n)-\lambda U_k(x_n))\to du_*-u_*(k_\yy -\lambda u_*)<du_*.\]
Let $\chi_{[-x_n,\yy)}(y)$ be the characteristic function of the set $[-x_n,\yy)$. Then, by Fatou's lemma,
  \bess
  \liminf_{n\to\yy}\int_{0}^{\yy}J(x_n-y)U_k(y)\dy
  &=&\liminf_{n\to\yy}\int_{-x_n}^{\yy}J(y)U_k(x_n+y)\dy\\
  &=&\liminf_{n\to\yy}\int_{-\yy}^{\yy}J(y)U_k(x_n+y)\chi_{[-x_n,\yy)}(y)\dy\\
  &\ge&\int_{-\yy}^{\yy}\liminf_{n\to\yy}J(y)U_k(x_n+y) \chi_{[-x_n,\yy)}(y)\dy\\
  &\ge&\int_{-\yy}^{\yy}J(y)u_*\dy=u_*.
  \eess
We get a contradiction, and so $u_*\ge k_\yy /\lambda$.

Assume $u^0:=\limsup_{x\to\yy}U_k(x)>k_\yy /\lambda$, and there is  $\tilde{x}_n\nearrow\yy$ such that $U_k(\tilde{x}_n)\to u^0$. So
  \[d\int_{0}^{\yy}J(\tilde x_n-y)U_k(y)\dy=dU_k(\tilde x_n)-U_k(\tilde x_n)(k(\tilde x_n)-\lambda U_k(\tilde x_n))\to du^0-u^0(k_\yy -\lambda u^0)>du^0\]
as $n\to\yy$. On the other hand, by the dominated convergence theorem,
\bess
  \lim_{n\to\yy}\int_{0}^{\yy}J(\tilde x_n-y)U_k(y)\dy&=&\lim_{n\to\yy}\int_{-\tilde x_n}^{\yy}J(y)U_k(\tilde x_n+y)\dy\\
  &=&\lim_{n\to\yy}\int_{-\yy}^{\yy}J(y)U_k(\tilde x_n+y) \chi_{[-\tilde x_n,\yy)}(y)\dy\\
  &\le&\lim_{n\to\yy}\int_{-\yy}^{\yy}J(y)\sup_{x\ge \tilde x_n+y}U_k(x) \chi_{[-\tilde x_n,\yy)}(y)\dy\\
  &=&\int_{-\yy}^{\yy}J(y)u^0\dy=u^0.
  \eess
We get a contradiction, and so $u^0\le k_\yy /\lambda$. Therefore, $\lim_{x\to\yy}U_k(x)=k_\yy /\lambda$.

We now prove the uniqueness. Let $v$ be another bounded positive solution of \eqref{2.8}. Similarly, $v$ is continuous and $v\le k_\yy /\lambda$. Similar to the proof of Lemma \ref{l2.4}, $v\ge \underline{u}$. Thanks to the above iteration process, it yields  $v\ge U_k$, and thus $\lim_{x\to\yy}v(x)=k_\yy /\lambda$. So similarly to the methods in Lemma \ref{l2.4}, we have $U_k\equiv v$. The uniqueness is obtained.

\sk{\rm(2)}\, By the conclusion (1), we have that $U_{k_i}$ is positive and continuous, and $\lim_{x\to\yy}U_{k_1}(x)=k_{1,\yy}/\lambda\le k_{2,\yy}/\lambda=\lim_{x\to\yy}U_{k_2}(x)$. Similarly to the proof of Lemma \ref{l2.4}, one can show that $U_{k_1}\le U_{k_2}$ in $\overline{\mathbb{R}}^+$, and the details are omitted here.
\end{proof}

\begin{remark}\label{r2.1}\, In Lemma \ref{l2.7}, if we further suppose that $k$ is strictly increasing in $\overline{\mathbb{R}}^+$, so is $U_k$. In fact, taking advantage of the analogous arguments in the proof of Lemma \ref{l2.4}, we can show that $U_k$ is nondecreasing in $\overline{\mathbb{R}}^+$. If $U_k(x_1)=U_k(x_2)$ for some $0<x_1<x_2$, it then follows that
\[\int_{0}^{\yy}J(x_1-y)U_k(y)\dy>\int_{0}^{\yy}J(x_2-y)U_k(y)\dy.\]
Thus
\[\int_{-x_1}^{\yy}J(y)U_k(y+x_1)\dy>\int_{-x_2}^{\yy}J(y)U_k(y+x_2)\dy,\]
which is impossible because $U_k$ is nondecreasing and positive in $\overline{\mathbb{R}}^+$.
\end{remark}

\begin{remark}\label{r2.2}
For any $0<\ep\ll1$, let $U_k^\ep$ be the unique bounded positive solution of \eqref{2.8} with $k(x)$ replaced by $k(x)\pm\ep$. By the boundness and monotone convergence of $U_k^\ep$, and uniqueness of bounded positive solution of \eqref{2.8}, we have that $U_k^\ep\to U_k$ as $\ep\to0$.
\end{remark}

\begin{lemma}\label{l2.8}\, Let $k$ satisfy the condition {\bf(K)}, and $U_k$ be the unique bounded positive solution of \eqref{2.8}. Assume $w_0\in C(\overline{\mathbb{R}}^+)\cap L^{\yy}(\mathbb{R}^+)$ and $w_0\ge,\not\equiv 0$. Then the problem
  \bess\left\{\begin{aligned}
&w_t= d\int_{0}^{\yy}J(x-y)w(t,y)\dy-dw+w(k(x)-\lambda w), && t>0,\ x\ge0,\\
&w(0,x)=w_0(x)\ge,\not\equiv0,  && x\ge0
 \end{aligned}\right.\eess
has a unique solution $w$, and $\lim_{t\to\yy}w(t,x)=U_k(x)$ locally uniformly in $\overline{\mathbb{R}}^+$.
\end{lemma}

The proof of Lemma \ref{l2.8} is similar to that of Lemma \ref{l2.5}, and we omit it here.

\begin{lemma}\label{l2.9}
Assume that $h(t)$ is continuous in $\overline{\mathbb{R}}^+$ and increasing to $\yy$, and $k$  satisfies the condition {\bf(K)} and $U_k$ is the unique bounded positive solution of \eqref{2.8}. Let $v$ be a nonnegative and bounded continuous function in $\overline{\mathbb{R}}^+\times\overline{\mathbb{R}}^+$ and satisfy
 \bes\label{2.9}
 \limsup_{t\to\yy}v(t,x)\le k(x) {\rm ~ ~ locally ~ uniformly ~ in ~ }\overline{\mathbb{R}}^+.
 \ees
Suppose that $u$ satisfies
\bess\left\{\begin{aligned}
&u_t= d\int_{0}^{h(t)}J(x-y)u(t,y)\dy-du+u(v(t,x)-\lambda u), && t>0,\ 0\le x<h(t),\\
&u(t,h(t))=0, &&t>0,\\
&h(0)=h_0,\;\; u(0,x)=u_0(x), && 0\le x\le h_0.
 \end{aligned}\right.\eess
Then $\limsup_{t\to\yy}u(t,x)\le U_k(x)$ locally uniformly in $\overline{\mathbb{R}}^+$.
 \end{lemma}

\begin{proof}\, The idea of this proof comes from \cite[Lemma 3.14]{DWZ}.
For any integer $n\ge1$, by \eqref{2.9}, there exist $T_n\nearrow\yy$ such that, for $t\ge T_n$ and $x\in[0,n+1]$, we have $v(t,x)\le k(x)+1/n$. Define
\begin{align*}
k_n(x)=\left\{\begin{aligned}
&k(x)+1/n,& &0\le x\le n,\\
&k(x)+1/n+2(K_0+1-k(x)-1/n)(x-n),& &n<x\le n+1/2,\\
&K_0+1, & &x>n+1/2,
\end{aligned}\right.
  \end{align*}
where $K_0>\max\{\|k\|_{\yy},\|v\|_{\yy}\}$.
It is clear that $k_n\in C(\overline{\mathbb{R}}^+)\cap L^{\yy}(\mathbb{R}^+)$, $k\le k_n\le K_0+1$, $k_n$ is nonincreasing in $n$, $\lim_{n\to\yy}k_n=k$ and $\lim_{x\to\yy}k_n(x)=K_0+1$. Fix $K_1>\max\{\|u_0\|_{\yy}+K_0+1, (K_0+1)/\lambda\}$. By Lemma \ref{l2.7}, there is $0<\ep\ll1$ such that $\ep U_k(x)<K_1$. Let $z_n$ be a solution of
  \bess\left\{\begin{aligned}
&z_t= d\int_{0}^{\yy}J(x-y)z(t,y)\dy-dz+z(k_n(x)-\lambda z), && t>T_n,\ x\ge0,\\
&z(T_n,x)=K_1,  && x\ge0.
 \end{aligned}\right.\eess
 It follows from a comparison argument that
 \bes\label{2.10}
 \ep U_k(x)\le z_n(t,x)\le K_1, ~ ~ u(t,x)\le z_n(t,x) {\rm ~ ~ ~ for ~ }t> T_n, ~ x\in\overline{\mathbb{R}}^+.
 \ees
 By virtue of Lemma \ref{l2.8}, we have
 \bes\label{2.11}
 \lim_{t\to\yy}z_n(t,x)=Z_n(x) {\rm ~ ~ ~ locally ~ uniformly ~ in ~} \overline{\mathbb{R}}^+,
 \ees
where $Z_n$ is the unique bounded positive solution of \eqref{2.8} with $k$ replaced by $k_n$. Since $k_n$ is nonincreasing in $n$, from Lemma \ref{2.7} and $z_n(t,x)\ge \ep U_k(x)$ we see that $Z_n$ is also nonincreasing in $n$ and $Z_n\ge\ep U_k$. Define $Z=\lim_{n\to\yy}Z_n$. Then $Z>0$ and satisfies \eqref{2.8}. The uniqueness implies $Z=U_k$. Thus combining \eqref{2.10} with \eqref{2.11}, we prove our conclusion.
\end{proof}

\section{The dynamics of the problem \eqref{1.2}}\lbl{s3}
\setcounter{equation}{0} {\setlength\arraycolsep{2pt}

Inspired by \cite[Theorem 2.1]{CDLL}, we can utilize the ODE theory, the contraction mapping principle and maximum principle to prove the global existence and uniqueness of solution of  problem \eqref{1.2}, and the details are omitted here.

\begin{theorem}[Existence and uniqueness]\label{t3.1}\, Problem \eqref{1.2} has a unique global solution $(u,h)$. Moreover, $u\in C(\overline D^T)$, $h\in C^1([0,T])$ and $0<u\le \max\{\|u_0\|_{\yy},K\}$ in $D^T_{h}$ for any $T>0$.
 \end{theorem}

Obviously, $h'(t)>0$ for $t>0$ by the third equation in \eqref{1.2}. Thus $h_\yy :=\lim_{t\to\yy}h(t)\le \yy$. In the following, we further suppose that $f$ satisfies conditions  {\bf(F2)} and {\bf(F3)}.

For positive constants $d$, $l$ and $a_0$, we define an operator $\mathcal{L}^d_{(0,\,l)}+a_0$ by
 \[\left(\mathcal{L}^d_{(0,\,l)}+a_0\right)\varphi:= d\left(\int_0^l J(x-y)
 \varphi(y)\dy-\varphi\right)+a_0\varphi\ \ \mbox{for}\;\; \phi\in C([0, l]).\]
The generalized principal eigenvalue of $\mathcal{L}^{d}_{(0,\,l)}+a_0$ is given by
 \[\lambda_p(\mathcal{L}^{d}_{(0,\,l)}+a_0):=\inf\{\lambda\in\mathbb R: (\mathcal{L}^{d}_{(0,\,l)}+a_0) \phi\leq\lambda\phi \mbox{ in $[0, l]$ for some } \phi\in C([0, l]), \;\phi>0\}.\]
Please see \cite[Proposition 3.4]{CDLL} for some useful properties of $\lambda_p(\mathcal{L}^{d}_{(0,\,l)}+a_0)$.
\begin{lemma}\label{l3.1}
If $h_{\yy}<\yy$, then $\lambda_p(\mathcal{L}^d_{(0,\,h_{\yy})}+f'(0))\le0$ and $\lim_{t\to\yy}\|u(t,\cdot)\|_{C([0,h(t)])}=0$.
\end{lemma}

\begin{proof} We first prove $\lambda_p(\mathcal{L}^d_{(0,\,h_{\yy})}+f'(0))\le0$. Arguing indirectly, suppose $\lambda_p(\mathcal{L}^d_{(0,\,h_{\yy})}+f'(0))>0$. By the condition {\bf(J)}, there are $\ep_0$, $\sigma_0>0$ such that $J(x)>\sigma_0$ for $|x|\le \ep_0$. Due to the continuity of $\lambda_p(\mathcal{L}^d_{(0,\,l)}+f'(0))$ about $l$, there exists $0<\ep<\ep_0/2$ such that $\lambda_p(\mathcal{L}^d_{(0,\,h_{\yy}-\ep)}+f'(0))>0$. For such $\ep$, there is $T>0$ such that $h(t)>h_{\yy}-\ep$ for all $t\ge T$. Let $w$ be the unique solution of
\bess\left\{\begin{aligned}
&w_t= d\int_{0}^{h_{\yy}-\ep}J(x-y)w(t,y)\dy-dw+f(w), && t>T,\ 0\le x\le h_{\yy}-\ep,\\
&w(T,x)=u(T,x)  && 0\le x\le h_{\yy}-\ep.
 \end{aligned}\right.\eess
In view of $\lambda_p(\mathcal{L}^d_{(0,\,h_{\yy}-\ep)}+f'(0))>0$ and \cite[Proposition 3.5]{CDLL}, we see that $w(t,x)$ converges to a positive steady state $W(x)$ uniformly in $[0,h_{\yy}-\ep]$ as $t\to\yy$. Moreover, it follows from a simple comparison argument that $u(t,x)\ge w(t,x)$ for $t\ge T$ and $x\in [0,h_{\yy}-\ep]$. So there exists $T_1>T$ such that $u(t,x)\ge \frac{1}{2} W(x)$ for $t\ge T_1$ and $x\in [0,h_{\yy}-\ep]$. Therefore, for $t>T_1$,
  \bess
h'(t)&=&\mu\dd\int_{0}^{h(t)}\int_{h(t)}^{\infty}
J(x-y)u(t,x)\dy\dx\ge\mu\dd\int_{h_{\yy}-\frac{\ep_0}{2}}^{h_{\yy}-\ep}\int_{h_{\yy}}^{h_{\yy}+\frac{\ep_0}{2}}
J(x-y)u(t,x)\dy\dx\\
&\ge&\mu\dd\int_{h_{\yy}-\frac{\ep_0}{2}}^{h_{\yy}-\ep}\int_{h_{\yy}}^{h_{\yy}+\frac{\ep_0}{2}}
\sigma_0 \frac{1}{2} W(x)\dy\dx>0,
\eess
which contradicts to $h_{\yy}<\yy$. And so $\lambda_p(\mathcal{L}^d_{(0,\,h_{\yy})}+f'(0))\le0$.

Next we prove $\lim_{t\to\yy}\|u(t,\cdot)\|_{C([0,h(t)])}=0$. Let $\bar{w}$ be the unique solution of
\bess\left\{\begin{aligned}
&\bar w_t= d\int_{0}^{h_{\yy}}J(x-y)\bar w(t,y)\dy-d\bar w+f(\bar w),\;\; t>0,\ 0\le x\le h_{\yy},\\
&\bar w(0,x)=u(0,x), ~ 0\le x\le h_0; ~ ~ \bar{w}(0,x)=0, ~ h_0\le x\le h_{\yy}.
 \end{aligned}\right.\eess
Since $\lambda_p(\mathcal{L}^d_{(0,\,h_{\yy})}+f'(0))\le0$, we see from \cite[Proposition 3.5]{CDLL} that $\lim_{t\to\yy}\bar{w}=0$ in $C([0,h_{\yy}])$.
By a comparison argument again, we get $u\le \bar{w}$ in $D^\yy_h$. Hence $\lim_{t\to\yy}\|u(t,\cdot)\|_{C([0,h(t)])}=0$.
\end{proof}

\begin{lemma}\label{l3.2}\, Let $U\in C(\overline{\mathbb{R}}^+)$ be the unique bounded positive solution of \eqref{2.4}. If $h_{\yy}=\yy$, then $\lim_{t\to\yy}u(t,x)=U(x)$ locally uniformly in $[0,\yy)$.
\end{lemma}

\begin{proof}\, Firstly, let $w$ be the unique solution of \eqref{2.6} with initial value $w(0,x)=\|u_0\|_{\yy}$. It follows from Lemma \ref{l2.5} that $\lim_{t\to\yy}w(t,x)=U(x)$ locally uniformly in $\overline{\mathbb{R}}^+$. By a comparison consideration, $u\le w$ in $D^\yy_h$. Thus $\limsup_{t\to\yy}u(t,x)\le U(x)$ locally uniformly in $\overline{\mathbb{R}}^+$.

By \cite[Proposition 3.4]{CDLL}, we know that for large $l$, $\lambda_p(\mathcal{L}^d_{(0,\,l)}+f'(0))>0$. For such large $l>0$, there exists $T>0$ such that $h(t)>l$ for $t\ge T$. Consider the following problem
\bess\left\{\begin{aligned}
&\underline{u}_t= d\int_{0}^{l}J(x-y)\underline{u}(t,y)\dy-d\underline{u}+f(\underline{u}), && t>T,\ 0\le x\le l,\\
&\underline{u}(T,x)=u(T,x)  && 0\le x\le l.
 \end{aligned}\right.\eess
Again, by the comparison principle, $u(t,x)\ge \underline{u}(t,x)$ for $t\ge T$ and $0\le x\le l$. Let $u_l$ be the unique positive solution of \eqref{x.3}. Then $\lim_{t\to\yy}\underline{u}(t,x)=u_l(x)$ uniformly in $[0,l]$, and $\lim_{l\to\yy}u_l=U$ locally uniformly in $\overline{\mathbb{R}}^+$. Hence, $\liminf_{t\to\yy}u(t,x)\ge U(x)$ locally uniformly in $\overline{\mathbb{R}}^+$. \end{proof}

Due to the above two lemmas, we immediately obtain the following conclusion.
\begin{theorem}[Spreading-vanishing dichotomy]\label{t3.2} One of the following alternatives happens for \eqref{1.2}:

\sk{\rm(1)}\, \underline{Spreading:} $h_{\yy}=\yy$ and $\lim_{t\to\yy}u(t,x)=U(x)$ locally uniformly in $\overline{\mathbb{R}}^+$;

\sk{\rm(2)}\, \underline{Vanishing:} $h_{\yy}<\yy$ and $\lim_{t\to\yy}\|u(t,\cdot)\|_{C([0,h(t)])}=0$.
\end{theorem}
Next we discuss when spreading or vanishing happens, and some criteria will be given. Thanks to Lemma \ref{l3.1} and \cite[Proposition 3.4]{CDLL}, we easily derive the following lemma.

\begin{lemma}\label{l3.3}
If $f'(0)\ge d$, then spreading happens for \eqref{1.2}.
\end{lemma}

If $f'(0)<d$, by \cite[Proposition 3.4]{CDLL} there is a unique $\ell^*_D>0$ such that $\lambda_p(\mathcal{L}^d_{(0,\,\ell^*_D)}+f'(0))=0$ and $\lambda_p(\mathcal{L}^d_{(0,\,l)}+f'(0))(l-\ell^*_D)>0$ for all $l>0$.

\begin{lemma}\label{l3.4}  Suppose $f'(0)<d$. If $h_0\ge\ell^*_D$, then spreading happens for \eqref{1.2}. Moreover, if $h_\yy<\yy$, then $h_\yy\le\ell^*_D$.
\end{lemma}

For later discussion, we now give two comparison principles which can be proved by analogous considerations with \cite[Theorem 3.1]{CDLL} and \cite[Lemma 3.4]{DN213}.

\begin{theorem}[Comparison principle]\label{t3.3}\, Let conditions {\bf (J)} and {\bf (F1)} hold, and $\bar{h}\in C^1([0,T])$, $\bar{u},\bar{u}_t\in C(\overline{D}^T_{\bar h})$ satisfy
\bes\label{3.1}
\left\{\begin{aligned}
&\bar u_t\ge d\dd\int_{0}^{\bar h(t)}J(x-y)\bar u(t,y)\dy-d\bar u+f(t,x,\bar u), && 0<t\le T,~0\le x<\bar h(t),\\
&\bar u(t,\bar h(t))\ge0,&& 0<t\le T,\\
&\bar h'(t)\ge\mu\dd\int_{0}^{\bar h(t)}\int_{\bar h(t)}^{\infty}
J(x-y)\bar u(t,x)\dy\dx,&& 0<t\le T,\\
&\bar h(0)\ge h_0,\;\; \bar u(0,x)\ge u_0(x),&& x\in[0,h_0].
\end{aligned}\right.
 \ees
 Then the solution $(u,h)$ of \eqref{1.2} satisfies
 \[u(t,x)\le \bar{u}(t,x), ~ ~ h(t)\le\bar{h}(t) ~ ~ {\rm for }~ 0<t\le T, ~ 0\le x\le h(t).\]
\end{theorem}

\begin{theorem}[Comparison principle]\label{t3.4}\, Let conditions {\bf (J)} and {\bf (F1)} hold, and $\bar{h}\in C^1([0,T])$, $r\in C([0,T])$ be nondecreasing and $0\le r(t)<\bar{h}(t)$. If $\bar{u},\bar{u}_t\in C(\overline{D}^T_{\bar{h}})$ satisfy
\bes\label{3.2}
\left\{\begin{aligned}
&\bar u_t\ge d\dd\int_{0}^{\bar h(t)}J(x-y)\bar u(t,y)\dy-d\bar u+f(t,x,\bar u), && 0<t\le T,~r(t)<x<\bar h(t),\\
&\bar u(t,\bar h(t))\ge0,&& 0<t\le T,\\
&\bar h'(t)\ge\mu\dd\int_{0}^{\bar h(t)}\int_{\bar h(t)}^{\infty}
J(x-y)\bar u(t,x)\dy\dx,&& 0<t\le T,\\
&\bar u(t,x)\ge u(t,x),&& 0\le t\le T, ~0\le x\le r(t),\\
&\bar h(0)\ge h_0,\;\; \bar u(0,x)\ge u_0(x),&& x\in[0, h_0],
\end{aligned}\right.
 \ees
then the solution $(u,h)$ of \eqref{1.2} satisfies
 \[u(t,x)\le \bar{u}(t,x), ~ ~ h(t)\le\bar{h}(t) ~ ~ {\rm for }~ 0<t\le T, ~ 0\le x\le h(t).\]
\end{theorem}

It can be easily seen that the above comparison principles are still valid if $\int_{\mathbb{R}} J(x)\dx \ne 1$ and condition {\rm(ii)} in {\bf (F1)} is not satisfied. The similar results hold for the solution $(u,h)$ of \eqref{1.3}. The pair $(\bar{u},\bar{h})$ is usually called an upper solution of \eqref{1.2} or \eqref{1.3}. We can also define a lower solution and derive analogous conclusions by reversing all the inequalities of \eqref{3.1} or \eqref{3.2}. Moreover, it follows from Theorem \ref{t3.3} that the solution $(u,h)$ of \eqref{1.2} or \eqref{1.3} is strictly increasing in $\mu>0$.

\begin{lemma}\label{l3.5} Let $f'(0)<d$. If $h_0<\ell^*_D$, then there exists $\underline{\mu}_D>0$ such that vanishing happens for \eqref{1.2} when  $0<\mu\le\underline{\mu}_D$.
\end{lemma}

\begin{proof} For any given $h_1\in (h_0,\ell^*_D)$, from \cite[Proposition 3.4]{CDLL} we know that $\lambda_1:=\lambda_p(\mathcal{L}^d_{(0,\,h_1)}+f'(0))<0$. Let $\phi_1$ be the corresponding positive eigenfunction to $\lambda_1$ with $\|\phi_1\|_{\yy}=1$. For constant $C>0$, we define $u_1=Ce^{\lambda_1 t/2}\phi_1$, and easily see that, for $t>0$ and $x\in[0,h_1]$,
\bess
u_{1t}-d\int_{0}^{h_1}J(x-y)u_1(t,y)\dy+du_1-f'(0)u_1=-\frac{\lambda_1}{2}Ce^{\lambda_1 t/2}\phi_1>0.
\eess
Let $\bar{u}$ be the unique solution of
\bess\left\{\begin{aligned}
&\bar u_t= d\int_{0}^{h_1}J(x-y)\bar u(t,y)\dy-d\bar u+f(\bar u), \quad t>0,\ 0\le x\le h_1,\\
&\bar u(0,x)=u(0,x)\;\;\mbox{in}\; [0, h_0], ~ ~ \bar{u}(0,x)=0\;\;\mbox{in}\; [h_0, h_1].
 \end{aligned}\right.\eess
Thanks to the assumptions on $f$, we have $f(\bar{u})\le f'(0)\bar{u}$. Choose $C>0$ such that $C\phi_1>u_0$ in $[0,h_0]$. Then, by the comparison principle,
 \[\bar{u}\le u_1=Ce^{\lambda_1 t/2}\phi_1\le Ce^{\lambda_1 t/2}\;\;\;{\rm in}\;\;\mathbb{R}^+\times[0,h_1].\]
 Define
 \[\bar{h}(t)=h_0+\mu h_1C\dd\int_{0}^{t}{\rm e}^{\lambda_1 s/2}{\rm d}s ~ ~ {\rm ~ for ~ ~ }t\ge0.\]
We now prove that $(\bar{u},\bar{h})$ is an upper solution of \eqref{1.2}. Clearly, for $t>0$,
 \[\bar{h}(t)=h_0-\frac{2\mu h_1C}{\lambda_1}(1-{\rm e}^{\lambda_1 t/2})<h_0-\frac{2\mu h_1C}{\lambda_1}\le h_1\]
 provided that
 \[0<\mu\le \underline{\mu}_D:=\frac{-\lambda_1(h_1-h_0)}{2h_1C}.\]
 Thus by the equation of $\bar{u}$, we obtain that, for $t>0$ and $x\in[0,\bar{h}(t))$,
 \[\bar u_t\ge d\int_{0}^{\bar{h}(t)}J(x-y)\bar u(t,y)\dy-d\bar u+f(\bar u).\]
Moreover, it is easy to check that
 \[
 \mu\dd\int_{0}^{\bar h(t)}\int_{\bar h(t)}^{\infty}
J(x-y)\bar u(t,x)\dy\dx\le \mu Ch_1{\rm e}^{\lambda_1 t/2}=\bar{h}'(t).
\]
Therefore, $(\bar{u},\bar{h})$ is an upper solution of \eqref{1.2}. By Theorem \ref{t3.3},
  \[u(t,x)\le\bar{u}(t,x), ~ ~ h(t)\le\bar{h}(t) ~ ~ {\rm for }~ ~ t>0, ~ x\in[0,h(t)].\]
Consequently, $h_{\yy}\le \lim_{t\to\yy}\bar{h}\le h_1$.
\end{proof}

From the above proof we easily see that, when $f'(0)<d$ and $h_0<\ell^*_D$, vanishing happens for \eqref{1.2} if $u_0(x)$ is small sufficiently.

\begin{lemma}\label{l3.6}\, Let $c\in C(\overline{\mathbb{R}}^+)\cap L^{\yy}(\mathbb{R}^+)$,  $g_0, H>0$, $w_0\in C([0,g_0])$ with $w_0(g_0)=0$ and $w_0>0$ in $[0,g_0)$. Then there exists $\mu^0>0$ depending on $J,d,c,w_0, g_0$ and $H$ such that when $\mu\ge\mu^0$ and $(w,g)$ satisfies
 \bess\left\{\begin{aligned}
&w_t\ge d\dd\int_{0}^{g(t)}J(x-y)w(t,y)\dy-dw+c(x)w,  \ \ t>0,~0\le x<g(t),\\
&w(t,g(t))=0,\ \ t>0,\\
&g'(t)\ge\mu\dd\int_{0}^{g(t)}\int_{g(t)}^{\infty}
J(x-y)w(t,x)\dy\dx,\ \ t>0,\\
&w(0,x)=w_0(x),~g(0)=g_0,\ \ x\in[0,g_0],
  \end{aligned}\right.\eess
we must have $\lim_{t\to\yy}g(t)\ge H$.
\end{lemma}

\begin{proof}\, By the maximum principle, $w>0$ and $g'>0$ for $t>0$ and $0\le x<g(t)$. Choose a function $s\in C^1([0,1])$ with $s(0)=g_0$, $s(1)=H$ and $s'>0$, and consider the problem
 \bess\left\{\begin{aligned}
&z_t= d\dd\int_{0}^{s(t)}J(x-y)z(t,y)\dy-dz+c(x)z,  \ \ t>0,~0\le x<s(t),\\
&z(t,s(t))=0,\ \ t>0,\\
&z(0,x)=w_0(x),~s(0)=g_0,\ \ x\in[0,g_0].
 \end{aligned}\right.\eess
Similar to the arguments in \cite[Lemma 2.3]{CDLL}, this problem has a unique solution $z\in C([0,1]\times[0,s(t)])$, and $z>0$ in $(0,1]\times[0, s(t)]$. Thus $\int_{0}^{s(t)}\int_{s(t)}^{\infty}J(x-y)z(t,x)\dy\dx$ is continuous and has a positive lower bound in $[0,1]$. Since $s\in C^1([0,1])$, there exists $\mu^0>0$ such that, for $\mu\ge\mu^0$,
\[s'(t)\le\mu\dd\int_{0}^{s(t)}\int_{s(t)}^{\infty}J(x-y)z(t,x)\dy\dx ~ ~ {\rm for }~ ~ t\in[0,1].\]
Comparing $(w,g)$ with $(z,s)$ yields $g\ge s$ in $[0,1]$. So $\lim_{t\to\yy}g(t)\ge s(1)=H$ since $g'(t)>0$.
\end{proof}

By Lemmas \ref{l3.4} and \ref{l3.6}, we easily have the following result which implies that when initial habitat $h_0$ and growth rate $f'(0)$ are small, spreading can happen if expanding rate $\mu$ is large enough. The details of proof are omitted here.

\begin{lemma}\label{l3.7} Suppose $f'(0)<d$. If $h_0<\ell^*_D$, then there exists $\bar{\mu}_D>0$ such that spreading happens for \eqref{1.2} when $\mu\ge\bar{\mu}_D$.
\end{lemma}

Based on the above lemmas, when initial habitat $h_0$ and growth rate $f'(0)$ are small, we may find a critical value of expanding rate $\mu$ by monotonicity and continuous dependence of solution on $\mu$, which governs spreading and vanishing for problem \eqref{1.2}.

\begin{lemma}\label{l3.8} Suppose $f'(0)<d$. If $h_0<\ell^*_D$, then there exists $\mu^*_D>0$ such that spreading happens for \eqref{1.2} when $\mu>\mu^*_D$, and vanishing happens when $0<\mu\le\mu^*_D$.
\end{lemma}

In conclusion, we have the spreading-vanishing criteria as below.
\begin{theorem}[Spreading-vanishing criteria]\label{t3.5} Let $(u,h)$ be the unique solution of \eqref{1.2}.
\begin{enumerate}\vspace{-2mm}
\item[{\rm(1)}]\, If $f'(0)\ge d$, then spreading happens;\vspace{-2mm}
\item[{\rm(2)}]\, If $f'(0)<d$, then there exists a unique $\ell^*_D>0$ such that spreading happens when $h_0\ge\ell^*_D$;\vspace{-2mm}
\item[{\rm(3)}]\, If $f'(0)<d$ and $h_0<\ell^*_D$, then there is $\mu^*_D>0$ such that spreading happens for \eqref{1.2} if and only if $\mu>\mu^*_D$.\vspace{-2mm}
\end{enumerate}
\end{theorem}

Next we discuss the spreading speed of \eqref{1.2}, and the following semi-wave problem is vital to our arguments.

\begin{proposition}[{\cite[Theorem 1.2]{DLZ}}]\label{p3.1}\, Assume that $J$ satisfies the condition {\bf(J)} and $f$ satisfies conditions {\bf(F1)}-{\bf(F3)}. Then the problem
\bes\label{3.3}\left\{\begin{array}{lll}
 d\dd\int_{-\yy}^{0}J(x-y)\phi(y)\dy-d\phi+c\phi'+f(\phi)=0, \quad -\yy<x<0,\\[1mm]
\phi(-\yy)=u^*,\ \ \phi(0)=0, \ \ c=\mu\dd\int_{-\yy}^{0}\int_{0}^{\yy}J(x-y)\phi(x)\dy\dx
 \end{array}\right.
 \ees
 has a unique solution pair $(c_0,\phi^{c_0})$ with $c_0>0$ and $\phi^{c_0}(x)$ nonincreasing in $(-\yy,0]$ if and only if the condition {\bf(J1)} is satisfied. We usually call $\phi^{c_0}$ the semi-wave solution of \eqref{3.3} with speed $c_0$.
\end{proposition}

\begin{theorem}[Spreading speed]\label{t3.6} Assume that {\bf(F1)}-{\bf(F3)} hold and spreading happens for \eqref{1.2}. Then we have
  \bes
\frac{U(0)}{u^*}c_0\le\liminf_{t\to\yy}\frac{h(t)}{t}\le\limsup_{t\to\yy}\frac{h(t)}{t}\le c_0 \;\;\;{\rm if~{\bf(J1)}~holds},\label{3.4}\\
 \lim_{t\to\yy}\frac{h(t)}{t}=\yy\;\;\;{\rm if~{\bf(J1)}~does~not~hold},\qquad\qquad\label{3.5}
  \ees
where $U$ is given by Lemma \ref{l2.4}.
\end{theorem}

\begin{proof}
{\bf Step 1}:\, {\it The proof of the last inequality of \eqref{3.4}}.
For any small $\ep>0$ and $L>0$ to be determined later, define
\[\bar{h}(t)=(1+\ep)c_0t+L, ~ ~ \bar{u}(t,x)=(1+\ep)\phi^{c_0}(x-\bar{h}(t)).\]
We are going to check that there exists $T>0$ such that
\bes\label{3.6}
\left\{\begin{aligned}
&\bar u_t\ge d\dd\int_{0}^{\bar h(t)}J(x-y)\bar u(t,y)\dy-d\bar u+ f(\bar u), && t>0,~0\le x<\bar h(t),\\
&\bar u(t,\bar h(t))\ge0,&& t>0,\\
&\bar h'(t)\ge\mu\dd\int_{0}^{\bar h(t)}\int_{\bar h(t)}^{\infty}
J(x-y)\bar u(t,x)\dy\dx,&& t>0,\\
&\bar h(0)\ge h(T),\;\; \bar u(0,x)\ge u(T,x),&& x\in[0,h(T)].
\end{aligned}\right.
 \ees
Clearly, from a simple comparison argument, we have $\limsup_{t\to\yy}u(t,x)\le u^*$ uniformly in $\overline{\mathbb{R}}^+$, which yields that there is $T>0$ such that $u(t,x)\le (1+\ep/2)u^*$ for $t\ge T$ and $x\in \overline{\mathbb{R}}^+$. Due to $\phi^{c_0}(-\yy)=u^*$, we may let $L$ large sufficiently such that $\bar h(0)=L>h(T)$ and $\bar{u}(0,x)=(1+\ep)\phi^{c_0}(x-L)\ge(1+\ep/2)u^*\ge u(T,x)$ for $x\in[0,h(T)]$.

Now we show that the first and third inequations of \eqref{3.6} hold. Direct calculations yield
 \bess
\bar u_t&=&-(1+\ep)^2c_0{\phi^{c_0}}'(x-\bar{h}(t))\ge -(1+\ep)c_0{\phi^{c_0}}'(x-\bar{h}(t))\\
&=&(1+\ep)\left(d\dd\int_{-\yy}^{\bar h(t)}J(x-y)\phi^{c_0}(y-\bar{h}(t))\dy-d\phi^{c_0}(x-\bar{h}(t))+f(\phi^{c_0}(x-\bar{h}(t)))\right)\\
&\ge&d\dd\int_{0}^{\bar h(t)}J(x-y)\bar u(t,y)\dy-d\bar u+f(\bar u).
\eess
Moreover,
\bess
(1+\ep)\mu\dd\int_{0}^{\bar h(t)}\int_{\bar h(t)}^{\infty}
J(x-y)\phi^{c_0}(x-\bar{h}(t))\dy\dx&=&(1+\ep)\mu\dd\int_{-\bar{h}(t)}^{0}\int_{0}^{\infty}
J(x-y)\phi^{c_0}(x)\dy\dx\\
&\le&(1+\ep)\mu\dd\int_{-\yy}^{0}\int_{0}^{\infty}
J(x-y)\phi^{c_0}(x)\dy\dx\\[1mm]
&=&(1+\ep)c_0=\bar{h}'(t).
\eess

By virtue of Theorem \ref{t3.3}, we have
 \[u(t+T,x)\le \bar u(t,x), ~ ~ h(t+T)\le \bar h(t) ~ ~ {\rm for } ~ t\ge 0, ~ x\in[0,h(t)].\]
 Thus
 \[\limsup_{t\to\yy}\frac{h(t)}{t}\le\lim_{t\to\yy}\frac{\bar{h}(t-T)}{t}=(1+\ep)c_0.\]
 Letting $\ep\to0$, we obtain the desired conclusion.

{\bf Step 2}: {\it The proof of first inequality of \eqref{3.4}}. Define
\begin{align*}
\xi(x)=\left\{\begin{aligned}
&1,& &|x|\le1,\\
&2-|x|,& &1\le|x|\le2,\\
&0, & &|x|\ge2,
\end{aligned}\right.
  \end{align*}
and $J_n(x)=\xi(\frac{x}{n})J(x)$. Clearly, $J_n$ are supported compactly and nondecreasing in $n$, and $J_n\le J$. Also we have $\lim_{n\to\yy}J_n(x)=J(x)$ in $L^1(\mathbb{R})$ and locally uniformly in $\mathbb{R}$.
Thus we can choose $n$ large enough, say $n\ge N>0$, such that $d(\|J_n\|_1-1)u+f(u)$ still meets conditions {\bf(F1)}-{\bf(F3)}.

 For any given $n\ge N$, let $(u_n,h_n)$ be the unique solution of
 \bes
\left\{\begin{aligned}
&u_{nt}=d\dd\int_{0}^{h_{n}(t)}J_n(x-y)u_n(t,y)\dy-du_n+f(u_n), && t>0,~0\le x<h_n(t),\\
&u_n(t,h_n(t))=0, && t>0,\\
&h_n'(t)=\mu\dd\int_{0}^{h_n(t)}\!\!\int_{h_n(t)}^{\infty}
J_n(x-y)u_n(t,x)\dy\dx, && t>0,\\
&h_n(0)=h(T),\;\; u_n(0,x)=u(T,x), && x\in[0,h(T)].
\end{aligned}\right.
 \lbl{3.y} \ees
 Since spreading happens for \eqref{1.2}, we can conclude from Theorem \ref{t3.5} that spreading happens for $(u_n,h_n)$ by choosing $T$ large enough.

(1)\, For any $n\ge N$, we claim that the semi-wave problem
\bes\label{3.7}\left\{\begin{array}{lll}
 d\dd\int_{-\yy}^{0}\!J_n(x-y)\phi_n(y)\dy-d\phi_n+c\phi_n'+ f(\phi_n)=0,\;\;-\yy<x<0,\\[2mm]
\phi_n(-\yy)=u^*_{n},\ \ \phi_n(0)=0, \ \ c_n=\mu\dd\int_{-\yy}^{0}\int_{0}^{\yy}\!J_n(x-y)\phi_n(x)\dy\dx
 \end{array}\right.
 \ees
has a unique solution pair $(c_n,\phi_n)$ with $c_n>0$ and $\phi_n$ nonincreasing, where  $u^*_{n}$ is the unique positive root of equation $d(\|J_n\|_1-1)u+f(u)=0$. Clearly, $u^*_{n}\le u^*_{n+1}\le u^*$ and $\lim_{n\to\yy}u_n^*=u^*$.

We may rewrite the above semi-wave problem as the equivalent form
 \bess\left\{\begin{array}{lll}
 d_n\dd\int_{-\yy}^{0}\!\hat J_n(x-y)\phi_n(y)\dy-d_n\phi_n+c\phi_n'+d(\|J_n\|_1-1)\phi_n+ f(\phi_n)=0,\;\;-\yy<x<0,\\[2mm]
\phi_n(-\yy)=u^*_{n},\ \ \phi_n(0)=0, \ \ c_n=\mu_n\dd\int_{-\yy}^{0}\int_{0}^{\yy}\!\hat J_n(x-y)\phi_n(x)\dy\dx,
 \end{array}\right.
 \eess
 where $d_n=d\|J_n\|_1$, $\hat J_n=J_n/\|J_n\|_1$ and $\mu_n=\mu \|J_n\|_1$. Since $J_n$ has a compact support, the condition {\bf(J1)} holds for $J_n$. Then our conclusion directly follows from \cite[Theorem 1.2]{DLZ}.

(2)\, We prove $\lim_{n\to\yy}c_n=c_0$ and $\lim_{n\to\yy}\phi_n=\phi^{c_0}$ locally uniformly in $(-\yy,0]$. Since $J_n$ is nondecreasing in $n$ and $J_n\le J$, we can see from \cite[Lemma 2.8]{DN21} that $c_{n}\le c_{n+1}\le c_0$ and $\phi_n\le\phi_{n+1}\le\phi^{c_0}$.
By the monotonicity and boundedness of $\phi_n$, there exists $\phi_{\yy}\ge\phi_N$ such that $\phi_n$ converges to $\phi_{\yy}$ in $(-\yy,0]$. Besides, due to $0\le\phi_n\le u^*$ and
 \[\sup_{n\ge N}|\phi'_n(x)|\le \frac{2du^*+\max_{u\in[0,u^*]}f(u)}{c_N}<\yy ~ ~ {\rm in ~}(-\yy,0],\]
in view of a consideration of compactness we have $\phi_{\yy}\in C(\mathbb{R})$ and $\lim_{n\to\yy}\phi_n=\phi_{\yy}$ locally uniformly in $(-\yy,0]$. Let $c_\yy=\lim_{n\to\yy}c_n\le c_0$. Next we show $c_\yy=c_0$ and $\phi_{\yy}=\phi^{c_0}$.

 For any given $x\in(-\yy,0)$, it can be seen from \eqref{3.7} that
 \[c_n\phi_n(x)-c_n\phi_n(0)=-\int_{0}^{x}\left(d\dd\int_{-\yy}^{0}\!J_n(z-y)\phi_n(y)\dy-d\phi_n(z)+f(\phi_n(z))\right){\rm d}z.\]
 Thanks to the dominated convergence theorem, we have
 \[c_\yy\phi_\yy(x)-c_\yy\phi_\yy(0)=-\int_{0}^{x}\left(d\dd\int_{-\yy}^{0}\!J(z-y)
 \phi_\yy(y)\dy-d\phi_\yy(z)+ f(\phi_\yy(z))\right){\rm d}z.\]
Therefore,
\[ d\dd\int_{-\yy}^{0}\!J(x-y)\phi_\yy(y)\dy-d\phi_\yy+c_\yy\phi_\yy'+ f(\phi_\yy)=0 ~ ~ {\rm in~ } (-\yy,0).\]
Combining this with $\phi_\yy(-\yy)\ge\phi_N(-\yy)>0$, we easily derive $\phi_\yy(-\yy)=u^*$. Moreover, by the monotone convergence theorem,
 \[c_\yy=\lim_{n\to\yy}c_n=\lim_{n\to\yy}\mu\dd\int_{-\yy}^0
 \int_{0}^{\yy}\!\!J_n(x-y)\phi_n(x)\dy\dx=\mu\dd\int_{-\yy}^0
 \int_{0}^{\yy}\!\!J(x-y)\phi_\yy(x)\dy\dx.\]
Using \cite[Theorem 1.2]{DLZ}, we get $c_\yy=c_0$ and $\phi_{\yy}=\phi^{c_0}$.

(3)\, Let $U\in C(\overline{\mathbb{R}}^+)$ be the unique bounded positive solution of \eqref{2.4}. For $n\ge N$, by Lemma \ref{l2.4}, we easily know that the steady state problem
 \bess
  d\dd\int_{0}^{\yy}J_n(x-y)U(y)\dy-dU+f(U)=0 {\rm ~ ~ ~ in ~ }\;\overline{\mathbb{R}}^+
  \eess
has a unique bounded positive solution $U_n\in C(\overline{\mathbb{R}}^+)$, which is non-decreasing in $\overline{\mathbb{R}}^+$. Moreover, $0<U_n<u^*_n$ and $\lim_{x\to\yy}U_n(x)=u^*_n$. Since $J_n\le J_{n+1}$, by similar considerations with Lemma \ref{l2.4}, we have $U_n\le U_{n+1}\le U$. Define $U_{\yy}=\lim_{n\to\yy}U_n$. Then  $U_{\yy}>0$ and is nondecreasing. By the dominated convergence theorem,
  \[d\dd\int_{0}^{\yy}J(x-y)U_{\yy}(y)\dy-dU_{\yy}+f(U_{\yy})=0 {\rm ~ ~ in ~ }\;\overline{\mathbb{R}}^+.\]
Then we easily deduce $U_{\yy}(\yy)=u^*$. By Lemma \ref{l2.4} and Dini's theorem,  $U_{\yy}\equiv U$ and $\lim_{n\to\yy}U_n=U$ locally uniformly in $\overline{\mathbb{R}}^+$.

(4)\, Now we prove the first inequality of \eqref{3.4}. For $0<\ep\ll 1$ and $L\gg 1$, we define
  \[\underline{h}(t)=(1-\ep)\frac{U_n(0)}{u^*_n}c_nt+2L, ~ ~ \underline{u}(t,x)=(1-\ep)\frac{U_n(0)}{u^*_n}\phi_n(x-\underline{h}(t)).\]
We will check that, for some $T_1>0$, $(\underline{u},\,\underline{h})$ satisfies
 \bes
\left\{\begin{aligned}
&\underline u_t\le d\dd\int_{0}^{\underline h(t)}J_n(x-y)\underline u(t,y)\dy-d\underline u+ f(\underline u), && t>0,~L<x<\underline h(t),\\
&\underline u(t,\underline h(t))\le0,&& t>0,\\
&\underline h'(t)\le\mu\dd\int_{0}^{\underline h(t)}\int_{\underline h(t)}^{\infty}
J_n(x-y)\underline u(t,x)\dy\dx,&& t>0,\\
&\underline u(t,x)\le u_n(t+T_1,x),&& t>0, ~0\le x\le L,\\
&\underline h(0)\le h_n(T_1),\;\;\underline u(0,x)\le u_n(T_1,x),&& x\in[0,\underline h(0)].
\end{aligned}\right.
 \lbl{3.x}\ees
As spreading happens for $(u_n,h_n)$, by Theorem \ref{t3.2}, there exists a $T_1>0$ such that $h_n(T_1)>2L=\underline h(0)$,
and $\underline{u}(t,x)\le (1-\ep)U_n(0)\le (1-\ep)U_n(x)\le u_n(t+T_1,x)$ for $t\ge 0$ and $x\in[0,2L]$ which implies $\underline u(0,x)\le u_n(T_1,x)$ in $[0,\underline h(0)]$.
Since $J_n(x)=0$ for $|x|\ge L$, we have
 \bess
\mu\dd\int_{0}^{\underline h(t)}\!\int_{\underline h(t)}^{\infty}\!
J_n(x-y)\underline u(t,x)\dy\dx&=&\mu(1-\ep)\frac{U_n(0)}{u^*_n}\dd\int_{0}^{\underline h(t)}\!\int_{\underline h(t)}^{\infty}\!J_n(x-y)\phi_n(x-\underline{h}(t))\dy\dx\\
&=&\mu(1-\ep)\frac{U_n(0)}{u^*_n}\dd\int_{-\underline h(t)}^{0}\!\int_{0}^{\infty}\!J_n(x-y)\phi_n(x)\dy\dx\\
&=&\mu(1-\ep)\frac{U_n(0)}{u^*_n}\dd\int_{-\yy}^{0}\!\int_{0}^{\infty}\!J_n(x-y)\phi_n(x)\dy\dx\\
&=&(1-\ep)\frac{U_n(0)}{u^*_n}c_n=\underline{h}'(t).
 \eess
Moreover, by denoting $z:=z(x,t)=x-\underline{h}(t)$, we have
\bess
\underline{u}_t&=&-(1-\ep)^2\left[\frac{U_n(0)}{u^*_n}\right]^2c_n\phi_n'(z)\le -(1-\ep)\frac{U_n(0)}{u^*_n}c_n\phi_n'(z)\\
&=&(1-\ep)\frac{U_n(0)}{u^*_n}\left(d_n\!\dd\int_{-\yy}^{\underline{h}(t)}\!\hat J_n(x-y)\phi_n(z(y,t))\dy-d_n\phi_n(z)+d(\|J_n\|_1-1)\phi_n(z)\right)\\
&&+(1-\ep) f(\phi_n(z))\\
&\le& d\dd\int_{0}^{\underline{h}(t)}J_n(x-y)\underline{u}(t,y)\dy
-d\underline{u}+d(\|J_n\|_1-1)\underline{u}+ f(\underline{u})\\
&\le&d\dd\int_{0}^{\underline{h}(t)}J_n(x-y)\underline{u}\dy-d\underline{u}+ f(\underline{u}).
 \eess
Hence, \eqref{3.x} holds. Applying Theorem \ref{t3.4} to \eqref{3.y} and \eqref{3.x}, we derive
$\underline h(t)\le h_n(t+T_1)$. As $J\ge J_n$, by Theorem \ref{t3.3}, $h(t+T)\geq h_n(t)$, and so $h(t+T+T_1)\geq h_n(t+T_1)\ge \underline h(t)$ for $n\ge N$. By the expression of $\underline h(t)$ and arbitrariness of $\ep$, it follows that
 \bes\label{3.9}
 \liminf_{t\to\yy}\frac{h(t)}{t}\ge \frac{U_n(0)}{u^*_n}c_n.
 \ees
Noticing $\lim_{n\to\yy}(U_n, u_n^*, c_n)=(U, u^*, c_0)$, by $n\to\yy$ in \eqref{3.9} we deduce the desired result.

{\bf Step 3}:\, {\it The proof of \eqref{3.5}}.
From the arguments of Step 2 we see that \eqref{3.9} still holds and $\lim_{n\to\yy}(U_n, u_n^*)=(U, u^*)$ when the condition {\bf(J1)} does not hold. Similar to the proof of conclusion (2)\,(ii) of \cite[Proposition 5.1]{DN213} we can show $\lim_{n\to\yy}c_n=\yy$ (the details are omitted here). Hence, by \eqref{3.9}, the conclusion \eqref{3.5} holds.\end{proof}

\section{Dynamics of the problem \eqref{1.3}}
\setcounter{equation}{0} {\setlength\arraycolsep{2pt}

Similarly to Section 3, one can directly derive analogous results on solution $(u,h)$ of problem \eqref{1.3}. So we just give these main conclusions, and omit the details of their proofs. Besides, we will pay our attention to the spreading speed for \eqref{1.3}.
\begin{theorem}[Existence and uniqueness]\label{t4.1}
  Problem \eqref{1.3} has a unique global solution $(u,h)$. Moreover, for any $T>0$, $u\in C(\overline D^T)$, $h\in C^1([0,T])$, and $0\le u\le \max\{\|u_0\|_{\yy},K\}$.
\end{theorem}

\begin{theorem}[Spreading-vanishing dichotomy]\label{t4.2} One of the following alternatives holds for \eqref{1.3}:

\sk{\rm(1)}\, \underline{Spreading:} $h_{\yy}=\yy$ and $\lim_{t\to\yy}u=u^*$ locally uniformly in $\overline{\mathbb{R}}^+$;

\sk{\rm(2)}\, \underline{Vanishing:} $h_{\yy}<\yy$ and $\lim_{t\to\yy}\|u(t,\cdot)\|_{C([0,h(t)])}=0$.
\end{theorem}
\begin{theorem}[Spreading-vanishing criteria]\label{t4.3} Let $(u,h)$ be the unique solution of \eqref{1.3}.
\begin{enumerate}\vspace{-2mm}
\item[{\rm(1)}]\, If $f'(0)\ge d/2$, then spreading happens;\vspace{-2mm}
\item[{\rm(2)}]\, If $f'(0)<d/2$, then there exists a unique $\ell^*_N>0$ such that spreading happens when $h_0\ge\ell^*_N$;\vspace{-2mm}
\item[{\rm(3)}]\, If $f'(0)<d/2$ and $h_0<\ell^*_N$, then there exists $\mu^*_N>0$ such that spreading happens if and only if $\mu>\mu^*_N$, where $\ell^*_N$ is determined by similar arguments with $\ell^*_D$ in Section 3. Indeed, it follows from Lemma \ref{l2.3} and $f'(0)< d/2$ that there exists a unique $\ell^*_N>0$ such that $\lambda_p(\mathcal{L}^N_{(0,\,\ell^*_N)}+f'(0))=0$ and $\lambda_p(\mathcal{L}^N_{(0,\,l)}+f'(0))(l-\ell^*_N)>0$ for all $l>0$.\vspace{-2mm}
\end{enumerate}
\end{theorem}

Now we study the spreading speed of problems \eqref{1.3}. The following theorem will be proved by several lemmas.
\begin{theorem}[Spreading speed]\label{t4.4}\, Let {\bf(F1)}-{\bf(F3)} hold and spreading happen for \eqref{1.3}. Then
 \begin{align*}
\lim_{t\to\yy}\frac{h(t)}{t}=\left\{\begin{aligned}
&c_0& &{\rm if~{\bf(J1)}~is~satisfied},\\
&\yy& &{\rm if~{\bf(J1)}~is~not~satisfied},
\end{aligned}\right.
  \end{align*}
  where $c_0$ is uniquely given by the semi-wave problem \eqref{3.3}.
\end{theorem}

\begin{lemma}\label{l4.1}Under the same assumptions with Theorem \ref{t4.4}, if the condition {\bf(J1)} holds then
  \bes
  \liminf_{t\to\yy}\frac{h(t)}{t}\ge c_0.
  \lbl{4.3}\ees
\end{lemma}
\begin{proof}

This proof is similar to Step 2 in the proof of Theorem \ref{t3.6}, but some obvious and crucial modifications are need. So we give the sketch.

{\bf Step 1}: Define $J_n(x)$ as in the proof of Theorem \ref{t3.6}. Clearly,
\[j_n(x):=\int_{-x}^{\yy}J_n(y)\dy\to j(x)=\int_{-x}^{\yy}J(y)\dy ~ ~ {\rm uniformly~for~}x\in\mathbb{R}{\rm ~as~} n\to\yy.\]
Moreover, there is $0<\delta_0\ll1$ such that $\tilde{f}(u):=-\delta u+f(u)$ satisfies conditions {\bf(F1)}-{\bf(F3)} for all $\delta\in(0,\delta_0)$, and has a unique positive zero $u^*_{\delta}$. For such $\delta$, we can choose $n$ large enough, say $n\ge N_1>0$, such that $d(\|J_n\|_1-1)u+\tilde{f}(u)$ still meets conditions {\bf(F1)}-{\bf(F3)} and
 \bes\label{4.1}
d(j_n(x)-j(x))+\delta\ge0 \;\;{\rm ~ for~all~}\; x\in \mathbb{R}.
 \ees
For any given $n\ge N_1$, let $(u_n,h_n)$ be the unique solution of
 \bess
\left\{\begin{aligned}
&u_{nt}=d\dd\int_{0}^{h_{n}(t)}J_n(x-y)u_n(t,y)\dy-dj_n(x)u_n+\tilde{f}(u_n), && t>0,~0\le x<h_n(t),\\
&u_n(t,h_n(t))=0, && t>0,\\
&h_n'(t)=\mu\dd\int_{0}^{h_n(t)}\!\!\int_{h_n(t)}^{\infty}
J_n(x-y)u_n(t,x)\dy\dx, && t>0,\\
&h_n(0)=h(T),\;\;u_n(0,x)=u(T,x), &&x\in[0,h(T)].
\end{aligned}\right.
 \eess
Analogously, we can conclude from Theorem \ref{t4.3} that spreading happens for $(u_n,h_n)$ by choosing $T$ large enough.

{\bf Step 2}: Similar to the proof of Theorem \ref{t3.6} ((1) of Step 2), the semi-wave problem
\bess\left\{\begin{array}{lll}
 d\dd\int_{-\yy}^{0}\!J_n(x-y)\phi_n(y)\dy-d\phi_n+c\phi_n'+\tilde f(\phi_n)=0,\;\;-\yy<x<0,\\[2mm]
\phi_n(-\yy)=u^*_{n,\delta},\ \ \phi_n(0)=0, \ \ c_n=\mu\dd\int_{-\yy}^{0}\int_{0}^{\yy}\!J_n(x-y)\phi_n(x)\dy\dx
 \end{array}\right.
 \eess
has a unique solution pair $(c_n,\phi_n)$ with $c_n>0$ and $\phi_n$ nonincreasing, where  $u^*_{n,\delta}$ is the unique positive root of $d(\|J_n\|_1-1)u+\tilde{f}(u)=0$. Clearly, $u^*_{n,\delta}\le u^*_{n+1,\delta}\le u^*_{\delta}$ and $\lim_{n\to\yy}u^*_{n,\delta}=u^*_{\delta}$.

{\bf Step 3}: By the same lines in the proof of Theorem \ref{t3.6} ((2) of Step 2), we have
$\lim_{n\to\yy}c_n=c_0^{\delta}$ and $\lim_{n\to\yy}\phi_n=\phi^{\delta}$ locally uniformly in  $(-\yy,0]$,
where $(c_0^{\delta},\phi^{\delta})$ is the unique solution pair of
 \bes\label{4.2}\left\{\begin{array}{lll}
 d\dd\int_{-\yy}^{0}\!J(x-y)\phi(y)\dy-d\phi+c\phi'+\tilde f(\phi)=0,\;\;-\yy<x<0,\\[2mm]
\phi(-\yy)=u^*_{\delta},\ \ \phi(0)=0, \ \ c=\mu\dd\int_{-\yy}^{0}\int_{0}^{\yy}\!J(x-y)\phi(x)\dy\dx.
 \end{array}\right.
 \ees

{\bf Step 4}: We now prove that $c_0^{\delta}\to c_0$ and $\phi^\delta\to \phi^{c_0}$ as $\delta\to 0$, where $(c_0,\phi^{c_0})$ is given by \eqref{3.3}. Let ${\delta_n}$ be a sequence decreasing to $0$, and $(c_0^{\delta_n},\phi^{\delta_n})$ be the unique solution pair of \eqref{4.2} with $\delta$ replaced by $\delta_n$. With the similar arguments in the proof of \cite[Lemma 2.8]{DN21}, we have $c_0^{\delta_n}\le c_0^{\delta_{n+1}}\le c_0$ and $\phi^{\delta_n}\le\phi^{\delta_{n+1}}\le \phi^{c_0}$. Then by arguing as in (2) of Step 2 in the proof of Theorem \ref{t3.6}, we can obtain the desired result.

{\bf Step 5}: We are going to prove $\liminf_{t\to\yy}\frac{h_n(t)}{t}\ge c_n$. For any given $0<\ep\ll1$ and $L$ large enough such that $J_n(x)=0$ for $|x|\ge L$, we define
\[\underline{h}(t)=c_n(1-\ep)t+2L, ~ ~ \underline{u}(t,x)=(1-\ep)\phi_n(x-\underline{h}(t)).\]
Analogously we easily examine that there exists $T_1>0$ such that
\bess
\left\{\begin{aligned}
&\underline u_t\le d\dd\int_{0}^{\underline h(t)}J_n(x-y)\underline u(t,y)\dy-dj_n(x)\underline u+\tilde f(\underline u), && t>0,~L<x<\underline h(t),\\
&\underline u(t,\underline h(t))\le0,&& t>0,\\
&\underline h'(t)\le\mu\dd\int_{0}^{\underline h(t)}\!\int_{\underline h(t)}^{\infty}
J_n(x-y)\underline u(t,x)\dy\dx,&& t>0,\\
&\underline u(t,x)\le u_n(t+T_1,x),&& t>0, ~0\le x\le L,\\
&\underline h(0)\le h_n(T_1),\;\;\underline u(0,x)\le u_n(T_1,x),&& x\in[0,\underline h(0)].
\end{aligned}\right.
 \eess
By Theorem \ref{t3.4}, we can get the desired result.

{\bf Step 6}:\, Now we prove \eqref{4.3}. Making use of $J_n\le J$ and \eqref{4.1} we have
\bess
u_t&=&d\dd\int_{0}^{h(t)}J(x-y)u(t,y)\dy-dj(x)u+f(u)\\
&=&d\dd\int_{0}^{h(t)}J(x-y)u(t,y)\dy-du(t,x)+d(1-j(x))u+f(u)\\
&\ge&d\dd\int_{0}^{h(t)}J_n(x-y)u(t,y)\dy-du+d(1-j_n(x))u+\tilde f(u).
\eess
 On the other hand,
\bess
u_{nt}&=&d\dd\int_{0}^{h_{n}(t)}J_n(x-y)u_n(t,y)\dy-dj_n(x)u_n+\tilde{f}(u_n)\\
&=&d\dd\int_{0}^{h_{n}(t)}J_n(x-y)u_n(t,y)\dy-du_n+d(1-j_n(x))u_n+\tilde{f}(u_n).
 \eess
By the comparison principle (Theorem \ref{t3.3}),
 \[u_n(t,x)\le u(t+T,x), ~ ~ h_n(t)\le h(t+T) ~ ~ {\rm for } ~ t\ge0, ~ x\in[0,h_n(t)].\]
Combining this with the results in Steps 3, 4 and 5 allows us to derive \eqref{4.3}.
  \end{proof}

\begin{lemma}\label{l4.2} Under assumptions of Theorem \ref{t4.4}, if the condition {\bf(J1)} holds then we have
\[\limsup_{t\to\yy}\frac{h(t)}{t}\le c_0.\]
\end{lemma}

\begin{proof} This lemma can be proved by similar arguments in Step 1 of Theorem \ref{t3.6}, and thus we only give the difference.
For any small $\ep>0$ and $L>0$ to be determined later, define
\[\bar{h}(t)=(1+\ep)c_0t+L, ~ ~ \bar{u}(t,x)=(1+\ep)\phi^{c_0}(x-\bar{h}(t)).\]
We are going to check that there exists $T>0$ such that
\bes\label{4.4}
\left\{\begin{aligned}
&\bar u_t\ge d\dd\int_{0}^{\bar h(t)}J(x-y)\bar u(t,y)\dy-dj(x)\bar u+ f(\bar u), && t>0,~0\le x<\bar h(t),\\
&\bar u(t,\bar h(t))\ge0,&& t>0,\\
&\bar h'(t)\ge\mu\dd\int_{0}^{\bar h(t)}\int_{\bar h(t)}^{\infty}
J(x-y)\bar u(t,x)\dy\dx,&& t>0,\\
&\bar u(0,x)\ge u(T,x),~\bar h(0)\ge h(T),&& x\in[0,h(T)].
\end{aligned}\right.
 \ees
We only show the first inequality of \eqref{4.4}. Since $\phi$ is nonincreasing in $(-\yy,0]$, one has
 \bess
\bar u_t&=&-(1+\ep)^2c_0{\phi^{c_0}}'(x-\bar{h}(t))\ge -(1+\ep)c_0{\phi^{c_0}}'(x-\bar{h}(t))\\
&=&(1+\ep)\left(d\dd\int_{-\yy}^{\bar h(t)}J(x-y)\phi^{c_0}(y-\bar{h}(t))\dy-d\phi^{c_0}(x-\bar{h}(t))+f(\phi^{c_0}(x-\bar{h}(t)))\right)\\
&\ge&d\dd\int_{0}^{\bar h(t)}J(x-y)\bar u(t,y)\dy-dj(x)\bar u+f(\bar u)+d\dd\int_{-\yy}^{0}J(x-y)\bar u(t,y)\dy\\
&&\qquad+d\dd\int_{0}^{\yy}J(x-y)\dy\bar u(t,x)-d\bar{u}(t,x)\\
&=&d\dd\int_{0}^{\bar h(t)}J(x-y)\bar u(t,y)\dy-dj(x)\bar u+f(\bar u)+d\dd\int_{-\yy}^{0}J(x-y)\left[\bar u(t,y)-\bar{u}(t,x)\right]\dy\\
&\ge&d\dd\int_{0}^{\bar h(t)}J(x-y)\bar u(t,y)\dy-dj(x)\bar u+f(\bar u).
\eess
By virtue of Theorem \ref{t3.3}, $h(t+T)\le \bar h(t)$ for $t\ge 0$. Thus
 \[\limsup_{t\to\yy}\frac{h(t)}{t}\le\lim_{t\to\yy}\frac{\bar{h}(t-T)}{t}=(1+\ep)c_0.\]
Letting $\ep\to0$, we complete the proof.
\end{proof}

\begin{lemma}\label{l4.3} Under assumptions of Theorem \ref{t4.4}, if the condition {\bf(J1)} is not true then
\[\lim_{t\to\yy}\frac{h(t)}{t}=\yy.\]
\end{lemma}

\begin{proof} The proof is similar to Step 3 in the proof of Theorem \ref{t3.6}, and we omit the details here.\end{proof}

Theorem \ref{t4.4} can be directly deduced from Lemmas \ref{l4.1}, \ref{l4.2} and \ref{l4.3}.

\section{The dynamics of problem \eqref{1.4}}
\setcounter{equation}{0} {\setlength\arraycolsep{2pt}

Now we are in a position to study the dynamics of problem \eqref{1.4}. Firstly, by following the analogous lines of \cite[Theorem 2.1]{DWZ} we can obtain the following well-posedness result.

\begin{theorem}[Existence and uniqueness]\label{t5.1}  Problem \eqref{1.4} has a unique global solution $(u_1,u_2,h)$. Moreover, $(u_1,u_2)\in [C(\overline D^T)]^2$, $h\in C^1([0,T])$, $0< u_1\le A_1:=\max\{\|u_{10}\|_{\yy},a_1/b_1\}$ and $0<u_2\le \max\{\|u_{20}\|_{\yy},(a_2+c_2A_1)/b_2\}$ in $D^T_{h}$ for any $T>0$.
\end{theorem}

\begin{lemma}\label{l5.1}Suppose that $J_i$ satisfy the condition {\bf(J)} and $J_i>0$ in $\mathbb{R}$. If $h_\yy<\yy$, then
\[\lim_{t\to\yy}\|u_i(t,\cdot)\|_{C([0,h(t)])}=0, ~ ~ ~   \lambda_p(\mathcal{L}^{d_i}_{(0,\,h_\yy)}+a_i)\le0.\]
\end{lemma}

\begin{proof} This lemma can be established by the approach in the proof of \cite[Theorem 3.3]{DWZ} with some obvious modifications, and we omit the details here.
\end{proof}

\begin{lemma}\label{l5.2}Assume that $a_1b_1b_2>a_2b_1c_1+a_1c_1c_2$. If $h_\yy=\yy$, then there exist four positive functions $\bar{u}_i$ and $\underline{u}_i$, $i=1,2$, such that
 \bess
 \underline{u}_1(x)\le \liminf_{t\to\yy}u_1(t,x)\le \limsup_{t\to\yy}u_1(t,x)\le\bar{u}_1(x) ~ ~ {\rm ~ locally ~ uniformly ~ in ~ } \overline{\mathbb{R}}^+,\\
\underline{u}_2(x)\le \liminf_{t\to\yy}u_2(t,x)\le \limsup_{t\to\yy}u_2(t,x)\le\bar{u}_2(x) ~ ~ {\rm ~ locally ~ uniformly ~ in ~ } \overline{\mathbb{R}}^+.\eess
\end{lemma}

\begin{proof} {\bf Step 1}:\, Let $w$ be the unique solution of
 \bes\left\{\begin{aligned}
& w_t= d_1\int_{0}^{\yy}J_1(x-y) w(t,y)\dy-d_1 w+w(a_1-b_1w),\;\;\; t>0,\ x\ge0,\\
& w(0,x)=u_{10}(x),  ~ ~  0\le x\le h_0; ~ ~ ~ w(0,x)=0,  ~ ~ ~  x\ge h_0.
 \end{aligned}\right.\lbl{5.x}\ees
The comparison principle gives $u_1\le w$ in $D^\yy_h$. On the other hand, in view of Lemma \ref{l2.4}, problem \eqref{5.x} has a unique bounded positive steady state $\bar{u}_1$ with $0<\bar{u}_1<a_1/b_1$ such that $\lim_{t\to\yy}w(t,x)=\bar{u}_1(x)$ locally uniformly in $\overline{\mathbb{R}}^+$. Hence $\limsup_{t\to\yy}u_1(t,x)\le\bar{u}_1(x)$ locally uniformly in $\overline{\mathbb{R}}^+$.

{\bf Step 2}:\, Obviously, $a_2+c_2\bar{u}_1$ satisfies the condition {\bf(K)} and $\limsup_{t\to\yy}[a_2+c_2u_1(t,x)]\le a_2+c_2\bar{u}_1(x)$ locally uniformly in $\overline{\mathbb{R}}^+$. By Lemma \ref{l2.9}, $\limsup_{t\to\yy}u_2(t,x)\le\bar{u}_2(x)$ locally uniformly in $\overline{\mathbb{R}}^+$, where $\bar{u}_2(x)$ is a unique bounded positive solution of
 \[  d_2\dd\int_{0}^{\yy}J_2(x-y)u(y)\dy-d_2u+u(a_2+c_2\bar{u}_1-b_2 u)=0 {\rm ~ ~ ~in ~ }\;\overline{\mathbb{R}}^+.\]

{\bf Step 3}:\, It can be learned from Lemmas \ref{l2.4} and \ref{l2.7} that $\bar{u}_2\le (a_2b_1+c_2a_1)/b_1b_2$ in $\overline{\mathbb{R}}^+$ and $\lim_{x\to\yy}\bar{u}_2(x)=(a_2b_1+c_2a_1)/b_1b_2$. Together with the assumption $a_1b_1b_2>a_2b_1c_1+a_1c_1c_2$, we can find $\sigma_1>0$ such that $a_1-c_1(\bar{u}_2+\sigma_1)$ has a positive lower bound $\beta_0$ in $\overline{\mathbb{R}}^+$. Thus there exists $L>0$ such that, for any $l>L$, $\lambda_p(\mathcal{L}^{d_1}_{(0,\,l)}+\beta_0)>0$. By similar arguments in the proof of Lemma \ref{l2.6} we easily see that, for any $\sigma\in(0,\sigma_1)$, the problem
   \[d_1\dd\int_{0}^{l}J_1(x-y)u(y)\dy-d_1u+u(a_1-c_1(\bar{u}_2(x)+\sigma)-b_1u)=0 {\rm ~ ~ ~ in ~ }\;[0,l]\]
has a unique positive solution, denoted by $\underline{u}_l$.

Moreover, for any $l>L$ and $\sigma\in(0,\sigma_1)$, there exists $T>0$ such that $h(t)>l$ and $u_2(t,x)\le \bar{u}_2(x)+\sigma$ for $t\ge T$ and $0\le x\le l$. Let $\underline u$ be the unique solution of
   \bess\left\{\begin{aligned}
& \underline u_t= d_1\int_{0}^{l}J_1(x-y) \underline u(t,y)\dy-d_1 \underline u+\underline u(a_1-c_1(\bar{u}_2(x)+\sigma)-b_1\underline u), && t>T,\ 0\le x\le l,\\
&\underline u(T,x)=u(T,x),  ~ ~  0\le x\le l.
 \end{aligned}\right.\eess
It is easy to show that $\lim_{t\to\yy}\underline u(t,x)=\underline{u}_l(x)$ uniformly in $[0,l]$ as $t\to\yy$ and $\lim_{l\to\yy}\underline{u}_l=\underline{u}^\sigma_1$ locally uniformly in $\overline{\mathbb{R}}^+$, where $\underline{u}^\sigma_1$ is the unique bounded positive solution of
 \bess
 d_1\dd\int_{0}^{\yy}J_1(x-y)u(y)\dy-d_1u+u(a_1-c_1(\bar{u}_2(x)+\sigma)-b_1 u)=0 {\rm ~ ~ in ~ }\;\overline{\mathbb{R}}^+.
 \eess
In view of Remark \ref{r2.2}, $\lim_{\sigma\to 0}\underline{u}^\sigma_1=\underline{u}_1$, where $\underline{u}_1$ is the unique bounded positive solution of
 \[d_1\dd\int_{0}^{\yy}J_1(x-y)u(y)\dy-d_1u+u(a_1-c_1\bar{u}_2(x)-b_1 u)=0 {\rm ~ ~ in ~ }\;\overline{\mathbb{R}}^+.\]
By the comparison principle, $u_1\ge \underline u$ for $t>T$ and $x\in[0,l]$. Therefore, $\liminf_{t\to\yy}u_1(t,x)\ge \underline{u}_1(x)$ locally uniformly in $\overline{\mathbb{R}}^+$.

Similarly, we can also show that $\liminf_{t\to\yy}u_2(t,x)\ge \underline{u}_2(x)$ locally uniformly in $\overline{\mathbb{R}}^+$, where $\underline{u}_2$ is a unique bounded positive solution of
 \[d_2\dd\int_{0}^{\yy}J_2(x-y)u(y)\dy-d_2u+u(a_2+c_2\underline {u}_1(x)-b_2 u)=0 {\rm ~ ~ in ~ }\;\overline{\mathbb{R}}^+.\]
 The proof is complete.
\end{proof}

Next we always suppose that $J_i$ satisfy the condition {\bf(J)} and $J_i>0$ in $\mathbb{R}$ for $i=1,2$, and then give the criteria to determine whether $h_\yy=\yy$. The following conclusions can be proved by similar approaches in \cite{DWZ}, and only some minor modifications are needed. Firstly, by Lemma \ref{l5.1}, we have $h_{\yy}=\yy$ if either $a_1\ge d_1$ or $a_2\ge d_2$.

Assume that $a_i<d_i$ for $i=1,2$, we easily see that there exist unique $\ell^*_i>0$ such that $\lambda_p(\mathcal{L}^{d_i}_{(0,\,\ell^*_i)}+a_i)=0$ and $\lambda_p(\mathcal{L}^{d_i}_{(0,\,\ell^*_i)}+a_i)(l-\ell^*_i)>0$ for all $l>0$. Hence, if $h_0\ge \ell^*:=\min\{\ell^*_1,\ell^*_2\}$, we must have $h_{\yy}=\yy$. Also we obtain that  $h_{\yy}<\yy$ implies $h_{\yy}\le\ell^*$.

Suppose that $a_i<d_i$ for $i=1,2$ and $h_0< \ell^*$. It can be analogously derived from Lemma \ref{l3.6} that there exists $\mu^*>0$ such that $h_\yy=\yy$ when $\mu_1+\mu_2\ge\mu^*$; Moreover, similarly to the proofs of Lemma \ref{l3.5} and \cite[Lemma 3.6]{DWZ}, we can find  $\mu_*>0$ such that $h_\yy<\yy$ when $\mu_1+\mu_2<\mu_*$.

To sum up, we obtain the following result concerned with when $h_{\yy}<\yy$.

\begin{theorem}\label{t5.2}\, Suppose that $J_i$ satisfy the condition {\bf(J)} and $J_i>0$ in $\mathbb{R}$ for $i=1,2$. Let $(u_1,u_2,h)$ be the unique solution of \eqref{1.4}. Then the followings hold:
\begin{enumerate}\vspace{-2mm}
\item[{\rm(1)}]\, if $a_1\ge d_1$ or $a_2\ge d_2$, then $h_\yy=\yy$;\vspace{-2mm}
\item[{\rm(2)}]\, if $a_i<d_i$ for $i=1,2$, then there exists a unique $\ell^*>0$ such that $h_\yy=\yy$ when $h_0\ge\ell^*$;\vspace{-2mm}
\item[{\rm(3)}]\, if $a_i<d_i$ for $i=1,2$ and $h_0<\ell^*$, then there exists $\mu^*\ge\mu_*>0$ such that $h_\yy=\yy$ when $\mu_1+\mu_2>\mu^*$, and $h_\yy<\yy$ when $\mu_1+\mu_2\le\mu_*$.\vspace{-2mm}
    \end{enumerate}
\end{theorem}

We now discuss the spreading speed and accelerated spreading for \eqref{1.4}. To this end, we first give a comparison principle which can be proved by similar methods in \cite[Theorem 3.1]{CDLL} and \cite[Lemma 4.1]{Wjde}. So the details are omitted here.

\begin{theorem}\label{t5.3} Let $J_i$ satisfy the condition {\bf(J)} for $i=1,2$. If for any $T>0$, $\bar{h}\in C^1([0,T])$, $\bar{u}_i,\,\bar{u}_{it}\in C(\overline{D}^T_{\bar{h}})$ and satisfy
\bess
\left\{\begin{aligned}
&\bar u_{1t}\ge d_1\dd\int_{0}^{\bar h(t)}J_1(x-y)\bar u_1(t,y)\dy-d_1\bar u_1+\bar u_1(a_1-b_1\bar{u}_1), && 0<t\le T,~0\le x<\bar{h}(t),\\
&\bar u_{2t}\ge d_2\dd\int_{0}^{\bar h(t)}J_2(x-y)\bar u_2(t,y)\dy-d_2\bar u_2+\bar u_2(a_2-b_2\bar{u}_2+c_2\bar{u}_1), && 0<t\le T,~0\le x<\bar{h}(t),\\
&u_i(t,\bar{h}(t))\ge0, &&0<t\le T,\\
&\bar h'(t)\ge\sum_{i=1}^2\dd\mu_i \int_{0}^{\bar h(t)}\!\int_{\bar h(t)}^{\infty}
J_i(x-y)\bar u_i(t,x)\dy\dx, &&0<t\le T,\\
&\bar h(0)\ge h_0,\;\;\bar u_i(0,x)\ge u_{i0}(x),&& x\in[0, h_0],
\end{aligned}\right.
 \eess
 then the solution $(u_1,u_2,h)$ of \eqref{1.4} satisfies
 \[u_1(t,x)\le\bar{u}_1(t,x), ~ u_2(t,x)\le \bar{u}_2(t,x), ~ h(t)\le\bar{h}(t) ~ ~ {\rm for} ~ t\ge0, ~ x\in[0,h(t)].\]
\end{theorem}
\begin{theorem}\label{t5.4} Under the same assumptions with Theorem \ref{t5.2}, the accelerated spreading happens for \eqref{1.4} if one of the followings holds:

\sk{\rm(1)}\, $J_1$ violates {\bf(J1)} and $a_1b_1b_2>a_2b_1c_1+a_1c_1c_2$;

\sk{\rm(2)}\,$J_2$ does not satisfy {\bf(J1)}.
\end{theorem}

\begin{proof}We only prove the conclusion (1) since the conclusion (2) can be proved by analogous methods. By simple comparison arguments, one easily derive that $\limsup_{t\to\yy}u_1(t,x)\le a_1/b_1$ and $\limsup_{t\to\yy}u_2(t,x)\le (a_2b_1+a_1c_2)/b_1b_2$ uniformly in $[0,\yy)$. Due to $a_1b_1b_2>a_2b_1c_1+a_1c_1c_2$, for any small $\ep>0$ with $a_1b_1b_2>a_2b_1c_1+a_1c_1c_2+b_1b_2\ep$, we can find some $T>0$ such that
\[u_2(t,x)\le \frac{a_2b_1+a_1c_2}{b_1b_2}+\ep\; ~ ~ {\rm for } ~ t\ge T, ~ x\in\overline{\mathbb{R}}^+.\]
Then $(u_1,h)$ satisfies
 \bess
\left\{\begin{array}{lll}
u_{1t}(t,x)\ge d_1\dd\int_{0}^{h(t)}J_1(x-y)u_1(t,y)\dy-d_1u_1\\[3mm]
\hspace{18mm}+u_1\left[a_1-c_1\left(\dd\frac{a_2b_1+a_1c_2}{b_1b_2}+\ep\right)-b_1u_1\right],\;& t>T,~0\le x<h(t),\\[2mm]
u(t,h(t))=0,& t>T,\\[2mm]
h'(t)\ge \mu_1\dd\int_{0}^{h(t)}\int_{h(t)}^{\infty}
J_1(x-y)u_1(t,x)\dy\dx,& t>T,\\[3mm]
u_1(T,x)=u_1(T,x),\ \ &x\in[0,h(T)].
\end{array}\right.
 \eess
Let $(\underline{u},\underline{h})$ be the unique solution of
   \bess
\left\{\begin{array}{lll}
\underline u_{t}= d_1\dd\int_{0}^{\underline h(t)}J_1(x-y)\underline u(t,y)\dy-d_1\underline u\\[3mm]
\hspace{18mm}+\underline u\left[a_1-c_1\left(\dd\frac{a_2b_1+a_1c_2}{b_1b_2}+\ep\right)-b_1\underline u\right],\;
&t>0,~0\le x<\underline h(t),\\[2mm]
\underline u(t,\underline h(t))=0,& t>0,\\[2mm]
\underline h'(t)=\mu_1\dd\int_{0}^{\underline h(t)}\int_{\underline h(t)}^{\infty}
J_1(x-y)\underline u(t,x)\dy\dx,& t>0,\\[3mm]
\underline h(0)=h(T_1),\ \underline u(0,x)=u_1(T_1,x), &x\in[0,h(T_1)].
\end{array}\right.
 \eess
Since $h_\yy=\yy$, from Theorem \ref{t3.5} we may choose $T_1>T$ large enough such that $\underline{h}(\yy)=\yy$. By comparison principle, we have that $\underline{h}(t)\le h(t+T_1)$. Noting that $J_1$ violates {\bf(J1)}, owing to Theorem \ref{t3.6}, $\lim_{t\to\yy}\frac{\underline{h}(t)}{t}=\yy$. So $\lim_{t\to\yy}\frac{h(t)}{t}=\yy$.
\end{proof}

Next we show that accelerated spreading cannot occur if $J_1$ and $J_2$ satisfy {\bf(J1)}.

\begin{theorem}\label{t5.5} Let $J_i$ satisfy the condition {\bf(J1)} for $i=1,2$. If $h_\yy=\yy$, then
 \bes
 \limsup_{t\to\yy}\frac{h(t)}{t}\le c_1+c_2,\lbl{5.a}\ees
where $c_i$ for $i=1,2$ are determined by the following two semi-wave problems respectively,
\bess\left\{\begin{array}{lll}
 d_1\dd\int_{-\yy}^{0}J_1(x-y)\phi_1(y)\dy-d_1\phi_1+c\phi'_1+\phi_1(a_1-b_1\phi_1)=0, \ \ -\yy<x<0,\\[3mm]
\phi_1(-\yy)=\dd\frac{a_1}{b_1},\ \ \phi_1(0)=0, \ \ c_1=\mu_1\dd\int_{-\yy}^{0}\int_{0}^{\yy}J_1(x-y)\phi_1(x)\dy\dx,
 \end{array}\right.
 \eess
and
\bess\left\{\begin{array}{lll}
 d_2\dd\int_{-\yy}^{0}J_2(x-y)\phi_2(y)\dy-d_2\phi_2+c\phi'_2
 +\phi_2\left(a_2+c_2\frac{a_1}{b_1}-b_2\phi_2\right)=0, \ \ -\yy<x<0,\\[3mm]
\phi_2(-\yy)=\dd\frac{a_2b_1+a_1c_2}{b_1b_2},\ \ \phi_2(0)=0, \ \ c_2=\mu_2\dd\int_{-\yy}^{0}\int_{0}^{\yy}J_2(x-y)\phi_2(x)\dy\dx.
 \end{array}\right.
 \eess
 \end{theorem}

\begin{proof}Denote the unique solution pairs of the above two problems by $(c_1,\phi_1)$ and $(c_2,\phi_2)$ respectively. For small $\ep>0$ and $L>0$ to be determined later, we define
 \[\bar{h}(t)=(1+\ep)(c_1+c_2)t+L, ~ \bar{u}_1=(1+\ep)\phi_1(x-\bar{h}(t)), ~ \bar{u}_2=(1+\ep)\phi_2(x-\bar{h}(t)).\]
We show that there exists $T>0$ such that $(\bar{u}_1,\bar{u}_2,\bar{h})$ satisfies
  \bes\label{5.1}
\left\{\begin{aligned}
&\bar u_{1t}\ge d_1\dd\int_{0}^{\bar h(t)}J_1(x-y)\bar u_1(t,y)\dy-d_1\bar u_1+\bar u_1(a_1-b_1\bar{u}_1), && t>0,~0\le x<\bar{h}(t),\\
&\bar u_{2t}\ge d_2\dd\int_{0}^{\bar h(t)}J_2(x-y)\bar u_2(t,y)\dy-d_2\bar u_2+\bar u_2(a_2-b_2\bar{u}_2+c_2\bar{u}_1), && t>0,~0\le x<\bar{h}(t),\\
&u_i(t,\bar{h}(t))\ge0, &&t>0,\\
&\bar h'(t)\ge\sum_{i=1}^2\dd\mu_i \int_{0}^{\bar h(t)}\!\int_{\bar h(t)}^{\infty}
J_i(x-y)\bar u_i(t,x)\dy\dx, &&t>0,\\
&\bar h(0)\ge h(T),\;\; \bar u_i(0,x)\ge u_{i}(T,x),&& x\in[0, h(T)].
 \end{aligned}\right.
 \ees
Firstly, there is $T>0$ such that
 \[u_1(t,x)\le (1+\frac{\ep}{2})\frac{a_1}{b_1}, ~\; u_2(t,x)\le (1+\frac{\ep}{2})\frac{a_2b_1+a_1c_2}{b_1b_2} ~ ~ {\rm for} ~ t\ge T, ~ x\in\overline{\mathbb{R}}^+.\]
We choose $L>0$ large sufficiently so that
 \[\bar u_1(0,x)=(1+\ep)\phi_1(x-L)\ge(1+\frac{\ep}{2})\frac{a_1}{b_1}\ge u_1(T,x)\;\;\;\mbox{in}\;\;[0,h(T)],\]
and
 \[\bar u_2(0,x)=(1+\ep)\phi_2(x-L)\ge (1+\frac{\ep}{2})\frac{a_2b_1+a_1c_2}{b_1b_2}\ge u_2(T,x)\;\;\;\mbox{in}\;\;[0,h(T)].\]
Moreover, by denoting $z:=z(t,x)=x-\bar{h}(t)$, we have
 \bess
 \bar{u}_{1t}&=&-(1+\ep)^2(c_1+c_2)\phi'_1(z)\ge -(1+\ep)c_1\phi'_1(z)\\
 &=&(1+\ep)\left(d_1\dd\int_{-\yy}^{\bar{h}(t)}J_1(x-y)\phi_1(z(t,y))\dy
 -d_1\phi_1(z)+\phi_1(z)(a_1-b_1\phi_1(z))\right)\\
 &\ge&d_1\dd\int_{0}^{\bar h(t)}J_1(x-y)\bar u_1(t,y)\dy-d_1\bar u_1+\bar u_1(a_1-b_1\bar{u}_1),
 \eess
 and
 \bess
 \bar{u}_{2t}&=&-(1+\ep)^2(c_1+c_2)\phi'_2(z)\ge -(1+\ep)c_2\phi'_2(z)\\
 &\ge&(1+\ep)\left(d_2\dd\int_{-\yy}^{\bar{h}(t)}\!\!J_2(x-y)\phi_2(z(t,y))\dy-d_2\phi_2(z)
 +\phi_2(z)(a_2+c_2\phi_1(z)-b_2\phi_2(z))\right)\\
 &\ge&d_2\dd\int_{0}^{\bar h(t)}J_2(x-y)\bar u_2(t,y)\dy-d_2\bar u_2+\bar u_2(a_2-b_2\bar{u}_2+c_2\bar{u}_1).
 \eess
 Clearly it remains to show fourth inequality in \eqref{5.1}. In fact,
 \bess
 \sum_{i=1}^2\mu_i \int_{0}^{\bar h(t)}\!\!\int_{\bar h(t)}^{\infty}\!
J_i(x-y)\bar u_i(t,x)\dy\dx&=&(1+\ep)\sum_{i=1}^2\mu_i \int_{0}^{\bar h(t)}\!\!\int_{\bar h(t)}^{\infty}\!J_i(x-y)\phi_i(x-\bar{h}(t))\dy\dx\\
&=&(1+\ep)\sum_{i=1}^2\mu_i \int_{-\bar{h}(t)}^{0}\int_{0}^{\infty}\!
J_i(x-y)\phi_i(x)\dy\dx\\
&\le&(1+\ep)\sum_{i=1}^2\mu_i \int_{-\yy}^{0}\int_{0}^{\infty}\!
J_i(x-y)\phi_i(x)\dy\dx\\
&=&(1+\ep)(c_1+c_2)=\bar{h}'(t).
 \eess
Therefore, \eqref{5.1} holds. By Theorem \ref{t5.3}, $\bar{h}(t)\ge h(t+T)$. Hence, $\limsup_{t\to\yy}\frac{h(t)}{t}\le \lim_{t\to\yy}\frac{\bar{h}(t)}{t}=(1+\ep)(c_1+c_2)$. The arbitrariness of $\ep$ implies the limit \eqref{5.a}.
  \end{proof}

\section{Discussion}

In this paper, we study three nonlocal diffusion models with a free boundary in one dimension space. Compared with the existing related works \cite{CDLL,DWZ,DLZ,DN21}, our models only have a free boundary, and the other one is fixed. More precisely, we assume that species can only expand their habitat through one side. On the other side, we suppose that once species cross the boundary, they will die promptly (models \eqref{1.2} and \eqref{1.4}), or that they can not jump through the boundary (model \eqref{1.3}).

For models \eqref{1.2} and \eqref{1.3}, we first establish the well-posedness and spreading-vanishing dichotomy, and then some criteria governing spreading and vanishing are given. In contrast to \cite{CDLL}, the solution component $u(t,x)$ of \eqref{1.2} will converge to a bounded positive function $U(x)$ as $t\to\yy$ locally uniformly in $[0,\yy)$ if $h_\yy=\yy$. Also we find that if $f'(0)\ge d/2$, then spreading happens for \eqref{1.3} which is different from the criteria of model \eqref{1.2} and that in \cite{CDLL}.

Moreover, by a approximate method we prove that \eqref{1.3} has a finite spreading speed if and only if {\bf(J1)} holds, which is analogous to that in \cite{DLZ}. However, when {\bf(J1)} is satisfied by $J$, only lower and upper bounds for spreading speed are obtained for model \eqref{1.2}. For the exact spreading speed, we leave it as a future work. Accelerated spreading occurs for \eqref{1.2} if {\bf(J1)} is violated.

At last, we extend partial conclusions in \cite{Wjde} to the nonlocal diffusion version \eqref{1.4} by considering a steady state problem on half space and a technical result. Besides, we prove that the accelerated spreading will happen if

\sk{\rm(1)}\, the diffusion kernel function for prey violates {\bf(J1)} and the weakly hunting condition holds, i.e. $a_1b_1b_2>a_2b_1c_1+a_1c_1c_2$, or

\sk{\rm(2)}\, the kernel function of predator does not satisfy {\bf(J1)}.

If kernel functions for prey and predator both satisfy {\bf(J1)}, then we obtain that accelerated spreading cannot happen. These results indicate that spreading speed of \eqref{1.4} is influenced by the kernel functions of two species, which is very different from local diffusion version in \cite{Wjde}.

For the classical Lotka-Volterra competition model with nonlocal diffusions and a free boundary, we expect that the parallel conclusions hold.


\begin{thebibliography}{99}
\bibliographystyle{siam}
\setlength{\baselineskip}{15pt}

\vspace{-1.5mm}\bibitem{CDLL}J.-F. Cao, Y.H. Du, F. Li and W.-T. Li, {\it The dynamics of a Fisher-KPP nonlocal diffusion model with free boundaries}. J. Funct. Anal., \textbf{277}(2019), 2772-2814.

\vspace{-1mm}\bibitem{HMV}V. Hutson, S. Martinez, K.Mischaikow and G. Vickers, {\it The evolution of dispersal}. J. Math. Biol., \textbf{47}(2003), 483-517.

\vspace{-1.5mm}\bibitem{AMRT}F. Andreu, J.M. Maz{\'o}n, J.D. Rossi and J. Toledo, {\it Nonlocal Diffusion Problems}. Math. Surveys Monogr., \textbf{165}, AMS, Providence, RI, 2010.

\vspace{-1.5mm}\bibitem{DL2010}Y.H. Du and Z.G. Lin, {\it Spreading-Vanishing dichotomy in the diffusive logistic model with a free boundary}. SIAM J. Math. Anal., \textbf{42}(2010), 377-405.

\vspace{-1.5mm}\bibitem{BDK}G. Bunting, Y.H. Du and K. Krakowski, {\it Spreading speed revisited: analysis of a free boundary model}. Netw.
Heterog. Media \textbf{7}(2012), 583-603.

\vspace{-1mm}\bibitem{Wjde}M.X. Wang, {\it On some free boundary problems of the Lotka-Volterra type prey-predator model}. J. Differential Equations, \textbf{256}(2014), 3365-3394.

\vspace{-1.5mm}\bibitem{DWZ}Y.H. Du, M.X. Wang and M. Zhao, {\it Two species nonlocal diffusion systems with free boundaries}. 2019, arXiv:1907.04542v1.

\vspace{-1.5mm}\bibitem{DLZ}Y.H. Du, F. Li and M.L. Zhou, {\it Semi-wave and spreading speed of the nonlocal Fisher-KPP equation with free boundaries}. 2019, arXiv:1909.03711.

\vspace{-1.5mm}\bibitem{DN21}Y.H. Du and W.J. Ni, {\it Semi-wave, traveling wave and spreading speed for monostable cooperative systems with nonlocal diffusion and free boundaries}. 2020, arXiv:2010.01244.

\vspace{-1.5mm}\bibitem{DN200}Y.H. Du and W.J. Ni, {\it Approximation of random diffusion equation by nonlocal diffusion equation in free boundary problems of one space dimension}. 2020, arXiv:2003.05560.

\vspace{-1.5mm}\bibitem{DN20}Y.H. Du and W.J. Ni, {\it Analysis of a West Nile virus model with nonlocal diffusion and free boundaries}. Nonlinearty, \textbf{33}(2020), 4407-4448.

\vspace{-1.5mm}\bibitem{DN213}Y.H. Du and W.J. Ni, {\it The high dimensional Fisher-KPP nonlocal diffusion equation with free boundary and radial symmetry}. 2021, arXiv:2021.05286v1.

\vspace{-1mm}\bibitem{WW1}J.P. Wang and M.X. Wang, {\it Free boundary problems with nonlocal and local diffusions I: global solution}. J. Math. Anal. Appl., \textbf{490}(2)(2020), 123974.

\vspace{-1.5mm}\bibitem{WW2}J.P. Wang and M.X. Wang, {\it Free boundary problems with nonlocal and local diffusions II: Spreading-vanishing and long-time behavior}. Discrete Contin. Dyn. Syst.B, \textbf{25}(12)(2020), 4721-4736.

\vspace{-1.5mm}\bibitem{LSW}L. Li, W.J. Sheng and M.X. Wang, {\it Systems with nonlocal vs. local diffusions and free boundaries}. J. Math. Anal. Appl., \textbf{483}(2)(2020), 123646.

\vspace{-1.5mm}\bibitem{LWW20}L. Li, J.P. Wang and M.X. Wang, {\it The dynamics of nonlocal  diffusion systems with different free boundaries}. Commun. Pure Appl. Anal., \textbf{19}(7)(2020), 3651-3672.

\vspace{-1.5mm}\bibitem{ZLZ}W.Y. Zhang, Z.H. Liu and L. Zhou, {\it Dynamics of a nonlocal diffusive logistic model with free boundaries in time periodic environment}. Discrete Contin. Dyn. Syst.B, DOI: 10.3934/dcdsb.2020256.

\vspace{-1.5mm}\bibitem{ZLD}M. Zhao, W-.T. Li and Y.H. Du, {\it The effect of nonlocal reaction in an epidemic model with nonlocal diffusion and free boundaries}. Comm. Pure Appl. Anal., \textbf{19}(2020), 4599-4620.

\vspace{-1.5mm}\bibitem{ZZLD}M. Zhao, Y. Zhang, W-.T. Li and Y.H. Du, {\it The dynamics of a degenerate epidemic model with nonlocal diffusion and free boundaries}. J. Differential Equations, \textbf{269}(2020), 3347-3386.

\vspace{-1.5mm}\bibitem{HCV}H. Berestycki, J. Coville and H. Vo, {\it On the definition and the properties of the principal eigenvalue of some nonlocal operators}. J. Funct. Anal.,
\textbf{271}(2016), 2701-2751.

\end{thebibliography}
\end{document}